\documentclass{conm-p-l}
\usepackage{amssymb}
\usepackage{amsthm}
\usepackage{graphicx, color, float}
\newtheorem{prop}{Proposition}

\def\bbR{\mathrm{I\!R}}
\def\bbZ{\mathsf{Z\hskip-4.2ptZ}}
\def\bz{\mathcal{B}}
\def\dz{\mathcal{D}}
\def\lz{\mathcal{L}}
\def\vz{\mathcal{V}}
\def\wz{\mathcal{W}}
\def\zz{\mathcal{Z}}

\def\hjt{\hs\widehat{\nnh\jt\nnh}\hs}

\def\hni{\,\widehat{\hskip-2pt N\nh}\hs}
\def\tni{\,\widetilde{\hskip-2pt N\nh}\hs}

\def\hyp{\hskip.5pt\vbox
{\hbox{\vrule width2.5ptheight0.5ptdepth0pt}\vskip2pt}\hskip.5pt}
\def\hs{\hskip.7pt}
\def\hh{\hskip.4pt}
\def\nh{\hskip-.7pt}
\def\nnh{\hskip-1.5pt}
\def\hn{\hskip-.4pt}
\def\w{^{\phantom i}}
\def\txm{{T\hskip-2.9pt_x\w\hn M}}
\def\tym{{T\hskip-2.9pt_y\w\hn M}}

\def\vg{\varGamma}

\def\jt{\Theta}
\def\vs{\varSigma}

\def\bbC{{\mathchoice {\setbox0=\hbox{$\displaystyle\mathrm{C}$}
\hbox{\hbox to0pt{\kern0.4\wd0\vrule height0.9\ht0\hss}\box0}} 
{\setbox0=\hbox{$\textstyle\mathrm{C}$}\hbox{\hbox 
to0pt{\kern0.4\wd0\vrule height0.9\ht0\hss}\box0}} 
{\setbox0=\hbox{$\scriptstyle\mathrm{C}$}\hbox{\hbox 
to0pt{\kern0.4\wd0\vrule height0.9\ht0\hss}\box0}} 
{\setbox0=\hbox{$\scriptscriptstyle\mathrm{C}$}\hbox{\hbox 
to0pt{\kern0.4\wd0\vrule height0.9\ht0\hss}\box0}}}}

\def\Lie{\pounds}
\def\lj{\langle}
\def\rg{\rangle}
\def\lr{\lj\cdot\hs,\nh\cdot\rg}

\newtheorem{theorem}{Theorem}[section]
\newtheorem{lemma}[theorem]{Lemma}

\theoremstyle{definition}

\theoremstyle{remark}
\newtheorem{remark}[theorem]{Remark}

\numberwithin{equation}{section}

\begin{document}

\title[]{Nijen\-huis geometry of parallel tensors}
\author[A. Derdzinski]{Andrzej Derdzinski}
\address[Andrzej Derdzinski]{Department of Mathematics\\
The Ohio State University\hskip-1pt\\
231 \hbox{W\hskip-1pt.} 18th Avenue\\
Columbus, OH 43210, USA}
\email{andrzej@math.ohio-state.edu}
\author[P. Piccione]{Paolo Piccione}
\address[Paolo Piccione]{Department of Mathematics\\
School of Sciences\\
Great Bay University\\
Dongguan, Guangdong 523000, China}
\address{{\it Permanent address\hh}: 
Departamento de Matem\'atica\\
Instituto de Matem\'atica e Esta\-t\'\i s\-ti\-ca\\
Uni\-ver\-si\-da\-de de S\~ao Paulo\\
Rua do Mat\~ao 1010, CEP 05508-900\\
S\~ao Paulo, SP, Brazil}
\email{paolo.piccione@usp.br}
\author[I.\ Terek]{Ivo Terek} 
\address[Ivo Terek]{Department of Mathematics\\
The Ohio State University\\
231 \hbox{W\hskip-1pt.} 18th Avenue\\
Columbus, OH 43210, USA} 
\address{{\it Current address\hh}:
Department of Mathematics and Statistics\\
Williams College, Wil\-li\-ams\-town, MA 01267, USA}
\email{it3@williams.edu}

\thanks{Research supported in part by a
FAPESP\nh-\hs OSU 2015 Regular Research Award (FAPESP grant: 2015/50265-6).
The third author was also supported by NSF DMS-2247747.}

\begin{abstract}A tensor -- meaning here a tensor field $\,\jt$
of any type $(p,q)$
on a manifold -- may be called in\-te\-gra\-ble if it is parallel relative
to some tor\-sion-free connection. We provide analytical and geometric 
characterizations of in\-te\-gra\-bi\-li\-ty for differential $q$-forms,
$q=0,1,2,n-1,n$ (in dimension $n$), vectors, bi\-vec\-tors, symmetric
$\,(2,0)\,$ and $\,(0,2)\,$
tensors, as well as com\-plex-di\-ag\-o\-nal\-iz\-able and 
nil\-po\-tent tensors of type $\,(1,1)$. In most cases,
in\-te\-gra\-bi\-li\-ty is equivalent to algebraic constancy of $\,\jt\,$
coupled with the vanishing of one or more suitably
defined Nijen\-huis-type tensors, 
depending on $\,\jt\,$ via a qua\-si\-lin\-e\-ar first-or\-der differential
operator. For $\,(p,q)=(1,1)$, they include the ordinary Nijen\-huis tensor.
\end{abstract}



\maketitle

\setcounter{section}{0}
\setcounter{theorem}{0}
\setcounter{prop}{3}
\renewcommand{\thetheorem}{\Alph{theorem}}
\renewcommand{\theprop}{\Alph{prop}}
\section{Introduction}
\setcounter{equation}{0}
We refer to a tensor field $\,\jt\,$ of any type on a manifold $\,M\,$ as 
{\it algebraically constant\/} when, for any $\,x,y\in M\nh$,
some linear iso\-mor\-phism $\,\txm\to\tym\,$ sends $\,\jt\nnh_x\w$ to
$\,\jt\nnh_y\w$. 
The algebraic constancy amounts to being constant for functions, to vanishing
nowhere or everywhere in the case of vector fields and $\,1$-forms, and to
having constant rank for symmetric or skew-sym\-me\-tric $\,(0,2)\,$ and
$\,(2,0)\,$ tensors.

We call a tensor field $\,\jt\,$ {\it in\-te\-gra\-ble\/}
if some tor\-sion-free connection makes it parallel, and {\it
locally constant\/} if it has constant components in suitable local
coordinates around each point. As one sees using a partition of unity, for
in\-te\-gra\-bi\-li\-ty of $\,\jt\,$ it suffices that such tor\-sion-free
connections exist locally. Consequently,
\begin{equation}\label{imp}
\mathrm{the\ local\ constancy\ of\ }\,\jt\,\mathrm{\ implies\ its\
in\-te\-gra\-bi\-li\-ty\ (but\ not\ conversely),}
\end{equation}
counterexamples to the converse being non\-flat 
pseu\-\hbox{do\hs-}Riem\-ann\-i\-an metrics.

Given an algebraically constant tensor $\,\jt\,$ on a manifold $\,M\,$ and a
distribution $\,\dz\subseteq T\nh M\,$ naturally associated with it, 
as $\,\dz\,$ is obviously $\,\nabla\nh$-par\-al\-lel when $\,\nabla\nh\jt=0$,
\begin{equation}\label{iii}
\mathrm{the\ in\-te\-gra\-bi\-li\-ty\ of\ }\,\jt\,\mathrm{\ implies\ the\ 
distribution}\hyp\mathrm{in\-te\-gra\-bi\-li\-ty\ of\ 
}\,\dz\hh.
\end{equation}
The local constancy of an algebraically constant tensor 
is nothing else than in\-te\-gra\-bi\-li\-ty, in the sense of
\cite[Prop.\,1.1]{s-kobayashi}, of the corresponding $\,G$-struc\-ture
(Remark~\ref{gstrc}).

With a $\,(1,1)\,$ tensor $\,\jt\,$ on a manifold $\,M\,$
one associates its Nijen\-huis tensor $\,N\nh$, introduced by
Nijen\-huis \cite{nijenhuis} and studied by many others 
\cite{bolsinov-konyaev-matveev,boubel,clark-bruckheimer,goel,grifone-mehdi,
hernando-reyes-gadea,e-kobayashi,kurita,thompson}, which sends 
vector fields $\,v,w\,$ to the vector field
\begin{equation}\label{nih}
N(v,w)\,=\,\jt[\jt v,w]\,+\,\jt[v,\jt w]\,
-\,[\jt v,\jt w]\,-\,\jt^2\hn[v,w]\hh.
\end{equation}
As pointed out by several authors
\cite[Sect.\,2.3]{clark-bruckheimer},
\cite[Definition\,2.2]{bolsinov-konyaev-matveev},
$\,N\nnh=\hh0\,$ identically whenever $\,\jt\,$ is in\-te\-gra\-ble since, 
for any tor\-sion-free
connection $\,\nabla\hn$ on $\,M\nh$,
\begin{equation}\label{tfr}
N(v,w)\,=\,[\jt\nabla_{\nnh\!v}\w\nnh\jt\,
-\,\nabla_{\!\!\jt\hn v}\w\jt]\hh w\,+
\,[\nabla_{\!\!\jt\hn w}\w\jt\,
-\,\jt\nabla_{\nh\!w}\w\jt]\hh v\hh.
\end{equation}
Various generalizations of the Nijen\-huis tensor have been proposed 
\cite{blaga-nannicini,kosmann-schwarzbach,reyes-tempesta-tondo,studeny,
tempesta-tondo}. Below, after stating Proposition~\ref{crbiv}, we
elaborate on such generalizations that are of interest to us and have
therefore been introduced in this paper.

Com\-plex-di\-ag\-o\-nal\-iz\-abil\-i\-ty of a linear en\-do\-mor\-phism of
a fi\-nite-di\-men\-sion\-al real vector space $\,V\hh$ means, as usual, 
di\-ag\-o\-nal\-iz\-abil\-i\-ty of its com\-plex-lin\-e\-ar extension to
the complexification of $\,V\nh$. Since any en\-do\-mor\-phism of 
$\,V\hh$ is, uniquely, the sum of a com\-plex-di\-ag\-o\-nal\-iz\-able and
a nil\-po\-tent one \cite[Sect.\,4.2]{humphreys}, it is natural to deal with
these two classes of en\-do\-mor\-phisms separately.

In Sect.\,\ref{io},\,\ref{nc}--\ref{sc},\,\ref{df},\,\ref{ss},\,\ref{st}
and\,\ref{ik}--\ref{tc} we prove our six main results, stated
below. We begin with a fact due to Kurita
\cite[Theorem 9]{kurita}, which also easily follows (see
Sect.\,\ref{cd}) from a theorem of 
Bol\-si\-nov, Ko\-nya\-ev and Mat\-ve\-ev
\cite[Theorem 3.2]{bolsinov-konyaev-matveev}:
\begin{remark}\label{cdiag}For an algebraically constant
com\-plex-di\-ag\-o\-nal\-iz\-able\/ $\,(1,1)\,$ tensor\/ $\,\jt\,$ on a 
manifold\/ $\,M\,$ of dimension\/ $\,n\ge1$, the vanishing of\/ $\,N$ is
equivalent to the in\-te\-gra\-bi\-li\-ty of\/ $\,\jt$, as well as to its
local constancy.
\end{remark}
Algebraically constant tensors $\,\jt\,$ of type $\,(1,1)\,$ give rise to the
vector sub\-bundles $\,\zz^i\nh=\mathrm{Ker}\,\jt^i$ and 
$\,\bz\hh^i\nh=\mathrm{Im}\,\jt^i$ of $\,T\nh M\nh$, for integers $\,i\ge0$.
\begin{theorem}\label{npjor}Given an algebraically constant nil\-po\-tent\/ 
$\,(1,1)\,$ tensor\/ $\,\jt\,$ on a manifold\/ $\,M\,$ of dimension\/
$\,n\ge1$, the following four conditions are equivalent.
\begin{enumerate}
\item[(i)] $N\nnh=\hh0\,$ 
and\/ $\,\zz^i\nh=\mathrm{Ker}\,\jt^i$ is in\-te\-gra\-ble for every\/
$\,i=1,\dots,n$.
\item[(ii)] In some commuting local frame\/ $\,e_1\w,\dots,e_n\w$ around each
point, $\,\jt\,$ has the Jor\-dan normal form, with\/ $\,\jt e_1\w=0\,$ and
$\,\jt e\hn_i\w=0\,$ or\/ $\,\jt e\hn_i\w= e\hn_{i-1}\w$ if\/ $\,i>1$.
\item[(iii)] $\jt\,$ is locally constant.
\item[(iv)] $\jt\,$ is in\-te\-gra\-ble.
\end{enumerate}
\end{theorem}
The Jor\-dan normal form of an algebraically constant 
nil\-po\-tent $\,(1,1)\,$ tensor $\,\jt\,$ may be represented by
\begin{equation}\label{wdi}
\mathrm{a\ weakly\ decreasing\ string\ }\,\,d_1\w\nnh\ldots d_m\w\mathrm{\ of\
positive\ integers,}
\end{equation}
each $\,d_q\w$ standing for a $\,d_q\w\nnh\times d_q\w$ Jor\-dan block matrix
with ones immediately above the diagonal and zeros everywhere else. Of
interest to
us are the Jor\-dan normal forms $\,d_1\w\nnh\ldots d_m\w$ such that
$\,d_1\w=\ldots=d_{m-1}\w$. In other words,
\begin{equation}\label{csd}
\begin{array}{l}
\mathrm{either\nh\ all\nh\ blocks\nh\ have\nh\ the\nh\ same\nh\ length,\nnh\
or\nh\ they\nh\ represent\
exactly}\\
\mathrm{two\ different\ lengths,\ with\ the\ shorter\ one\ occurring\ just\
once.}
\end{array}
\end{equation}
We say that a given algebraic type of an algebraically constant
nil\-po\-tent $\,(1,1)$ tensor $\,\jt\,$ is {\it controlled by the
Nijen\-huis tensor\/} if the vanishing of $\,N$ implies, on any
underlying manifold, the local constancy of $\,\jt$.
\begin{theorem}\label{cntrl}Condition\/ {\rm(\ref{csd})} imposed on the
Jor\-dan normal form of an algebraically constant nil\-po\-tent\/ $\,(1,1)\,$ 
tensor\/ $\,\jt\,$ on a manifold\/ $\,M\,$ of dimension\/ $\,n\ge1\,$ is
necessary and sufficient for the algebraic type of\/ $\,\jt\hh$ to
be controlled by its Nijen\-huis tensor.
\end{theorem}
Theorem~\ref{cntrl} would be true as stated even if our definition of being
controlled by $\,N\hs$ referred to in\-te\-gra\-bi\-li\-ty rather than
local constancy. Namely, Proposition~\ref{ctrol} -- our proof of the necessity
of (\ref{csd}) -- realizes any $\,\jt\,$ not satisfying (\ref{csd}) as a
left-in\-var\-i\-ant tensor with $\,N\nnh=\hh0\,$ on a step $\,2\,$
nil\-po\-tent Lie group, which
fails the in\-te\-gra\-bi\-li\-ty test (\ref{imp}) due to having
non\-in\-te\-gra\-ble $\,\mathrm{Ker}\,\jt\hh^p$ for some integer $\,p\ge1$.

For nil\-po\-tent $\,(1,1)\,$ tensors $\,\jt\,$ which are {\it generic},
that is, $\,\,\dim\hs\mathrm{Ker}\,\jt=1$ or, equivalently, the 
Jor\-dan normal form of $\,\jt\,$ is the one-term string $\,n\,$ (a single
Jor\-dan block), the sufficiency of (\ref{csd}) in Theorem~\ref{cntrl} is a
result of Ko\-ba\-ya\-shi \cite[Sect.\,3]{e-kobayashi}. See also
\cite[Cor.\,2.4]{grifone-mehdi}, \cite[Theorem\,1]{thompson}, 
\cite[Theorem\,1.3, Cor.\,1.5]{boubel},
\cite[Theorem\,4.6]{bolsinov-konyaev-matveev}. Ko\-ba\-ya\-shi
\cite[Sect.\,5]{e-kobayashi} further illustrated the necessity of (\ref{csd})
by an example, with $\,n=4\,$ and the Jor\-dan normal form $\,211$, cited in
\cite[Example\,2.1]{bolsinov-konyaev-matveev}.

Sect.\,\ref{ga} exhibits a special case of Theorem~\ref{cntrl} by means of
an af\-fine-bun\-dle construction, resulting in
nonzero algebraically constant nil\-po\-tent
$\,(1,1)\,$ tensors $\,\jt\,$ with $\,N\nnh=\hh0$, satisfying the condition
$\,\jt^2\nh=0\,$ (equivalent to 
$\,\mathrm{Im}\,\jt\subseteq\mathrm{Ker}\,\jt$, that is, 
to having the Jor\-dan normal
form $\,2\ldots2\,$ or $\,2\ldots21\ldots1$).

For the normal form $\,2\ldots2$, 
also characterized by the equality 
$\,\mathrm{Ker}\,\jt=\mathrm{Im}\,\jt$, corresponding to the {\it
al\-most-tan\-gent 
structures\/} \cite{yano-davies}, the assertion of 
Theorem~\ref{cntrl} is due to Goel \cite[Theorem\,2.4]{goel}, while our
af\-fine-bun\-dle construction becomes that of Crampin and Thompson
\cite{crampin-thompson}. Our construction is ``locally universal''
(Theorem~\ref{sqezr}), which generalizes the local version of
\cite[Theorem on p.\,69]{crampin-thompson}.

We justify the following observation in Sect.\,\ref{df}.
\begin{prop}\label{dffrm}The closedness of an algebraically constant 
differential\/ $\,q$-form on an\/ 
$\,n$-man\-i\-fold, $\,q=0,1,2,n-1,n$, implies its local constancy.
\end{prop}
The converse implication (closedness from in\-te\-gra\-bi\-li\-ty,
and hence also from local constancy) is obviously true for forms of all
degrees.

Even weakened by the
replacement of local constancy with in\-te\-gra\-bi\-li\-ty -- cf.\
(\ref{imp}) --  Proposition~\ref{dffrm} fails to hold for differential forms
of other degrees: as we verify in 
Sect.\,\ref{do}, for any dimension $\,n\ge5\,$ and any
$\,q\in\{3,\dots,n-2\}$, in local coordinates $\,x^1\nh,\dots,x^n\nh$, the
following formula defines a
differential $\,q$-form $\,\zeta\hh$ which
is algebraically constant and closed, but not in\-te\-gra\-ble:
\begin{equation}\label{nin}
\zeta=
(dx^1\nnh\wedge dx^2\nh+dx^3\nnh\wedge dx^4)\wedge(dx^5\nnh+x^1dx^2\nh
-x^3dx^4)\wedge dx^6\nnh\wedge\ldots\wedge dx\hh^{q+2}\nh.
\end{equation}
Con\-stant-rank (skew)sym\-me\-tric $\,(0,2)\,$ and $\,(2,0)\,$ tensors,
being bundle mor\-phisms
$\,T\nh M\to T^*\hskip-1.8ptM\,$ or $\,T^*\hskip-1.8ptM\to T\nh M\nh$, have 
well-de\-fin\-ed unique kernels and images. The next displayed condition uses
the natural concept of pro\-ject\-a\-bil\-i\-ty, presented in
Sect.\,\ref{pr}: for
the in\-te\-gra\-bi\-li\-ty of a con\-stant-rank symmetric $\,(0,2)$ tensor 
$\,g\,$ on a manifold, it is necessary and sufficient -- as we justify in
Sect.\,\ref{ss} -- that
\begin{equation}\label{nas}
\mathrm{the\ distribution\ }\,\mathrm{Ker}\,g\,\mathrm{\ be\
in\-te\-gra\-ble,\ and\ }\,g\,\mathrm{\ pro\-ject\-a\-ble\ along\
}\,\mathrm{Ker}\,g\hh.
\end{equation}
Condition (\ref{nas}), rephrased as $\,\Lie\hskip-1.2pt_v\w g=0\,$ for every
local section $\,v\,$ of $\,\mathrm{Ker}\,g$, is well known to be an
in\-te\-gra\-bi\-li\-ty test for $\,g$. To the best of our knowledge, this
fact goes back to Moisil \cite{moisil} and Vr\u anceanu \cite{vranceanu}. See
also \cite{crampin,oproiu,sulik}, \cite[Theorem\,5.1]{duggal-bejancu}.

The sweeping recent result of Bandyopadhyay, Dacorogna, Mat\-ve\-ev and
Tro\-ya\-nov 
\cite[Theorem 4.4]{bandyopadhyay-dacorogna-matveev-troyanov} provides a
characterization of local constancy for $\,(0,2)$ tensors {\it without any 
sym\-me\-try$\hh/\nnh$skew-sym\-me\-try assumptions}. The criterion
(\ref{nas}), much more modest in scope, focuses on the symmetric case and
in\-te\-gra\-bi\-li\-ty (as opposed to local constancy); what
we gain is simplicity of the resulting conditions.

In contrast with $\,1$-forms (Proposition~\ref{dffrm}), the local constancy of
a vector field obviously follows just from its algebraic constancy. An
analogous
difference occurs between symmetric $\,(2,0)\,$ tensors and symmetric
$\,(0,2)$ tensors: 
the former -- unlike the latter -- require no pro\-ject\-a\-bil\-i\-ty 
condition to guarantee in\-te\-gra\-bi\-li\-ty.
\begin{prop}\label{inttz}The in\-te\-gra\-bi\-li\-ty of 
a con\-stant-rank symmetric\/ 
$\,(2,0)\,$ tensor\/ $\,\jt\,$ on a manifold is equivalent
to the in\-te\-gra\-bi\-li\-ty of\/ $\,\mathrm{Im}\,\hs\jt$.
\end{prop}
Such tensors $\,\jt\,$ can be naturally identified (see Remark~\ref{sbpsr})
with {\it \hbox{sub\hh-\hn}pseu\-\hbox{do\hs-}Riem\-ann\-i\-an
metrics\/} \cite{grochowski-krynski}, which include the 
\hbox{sub\hh-\hn}Riem\-ann\-i\-an ones \cite{bejancu}.

For a bi\-vec\-tor, that is, a skew-sym\-me\-tric $\,(2,0)\,$ tensor $\,\jt$,
assumed to have constant rank, formula (\ref{rst}) defines the {\it
restriction\/} of $\,\jt\hs$ to $\,\bz=\mathrm{Im}\,\hs\jt$, which is a
nondegenerate section of $\,\bz^{\wedge2}\nh$, thus giving rise to its
inverse, a section of $\,[\bz^*]^{\wedge2}\nh$.
\begin{prop}\label{crbiv}A con\-stant-rank bi\-vec\-tor\/ $\,\jt\,$ on a
manifold is locally constant or -- equivalently -- in\-te\-gra\-ble if and
only if the distribution\/
$\,\mathrm{Im}\,\hs\jt\hs$ is in\-te\-gra\-ble and the inverse of the
restriction of\/ $\,\jt\,$ to\/ $\,\mathrm{Im}\,\hs\jt\hs$ is closed along each
leaf of\/ $\,\mathrm{Im}\,\hs\jt$.
\end{prop}
The generalizations of the Nijen\-huis tensor which are of interest to us are 
motivated by Remark~\ref{cdiag}, Theorem~\ref{cntrl} and
Proposition~\ref{dffrm}:
we want to associate with a given tensor $\,\jt\hs$ one (or more)
Nijen\-huis-type tensor(s), each depending on $\,\jt\hs$ via a
qua\-si\-lin\-e\-ar 
first-or\-der differential operator, in such a way that, if $\,\jt\hs$
algebraically constant, the vanishing of these tensors is equivalent to
the in\-te\-gra\-bi\-li\-ty of $\,\jt$.

As an example, $\,N\hs$ given by (\ref{nih}) serves in this capacity for
com\-plex-di\-ag\-o\-nal\-iz\-able $\,(1,1)\,$ tensors and nil\-po\-tent
$\,(1,1)\,$ tensors with the property (\ref{csd}); its
qua\-si-lin\-e\-ar\-i\-ty is immediate from (\ref{tfr}). In the case of
differential $\,q$-forms $\,\zeta\hh$ in dimension $\,n$, where
$\,q\in\{0,1,2,n-1,n\}\,$ (but not -- see (\ref{nin})  -- for other degrees),
the exterior derivative $\,d\hh\zeta\hh$ is a 
Nijen\-huis-type tensor in our sense, while an analogous role for vector
fields is played by the zero tensor.

For any symmetric $\,(0,2)\,$ tensor $\,g\,$ of constant rank $\,r\,$ on a
manifold $\,M\nh$, we introduce two Nijen\-huis-type tensors
$\,N'$ and $\,N''\nh$, both of type $\,(0,2r+3)$, defined as follows:
$\,N'$ (or, $\,N''$)
sends vector fields $\,v,v_1\w,\dots,v_r\w$ (or, $\,w,u,v_1\w,\dots,v_r\w$)
to the $\,(r+2)$-form, or $\,(r+1)$-form
\begin{equation}\label{nnh}
\begin{array}{rl}
\mathrm{a)}&N'\nh(v,v_1\w,\dots,v_r\w)
=d\hs[g(v,\,\cdot\,)]\wedge g(v_1\w,\,\cdot\,)\wedge\ldots\wedge
g(v_r\w,\,\cdot\,)\hh,\\
\mathrm{b)}&N''\nh(w,u,v_1\w,\dots,v_r\w)
=\{[\Lie\nh g](w,u)\}\wedge g(v_1\w,\,\cdot\,)\wedge\ldots\wedge
g(v_r\w,\,\cdot\,)\hh,
\end{array}
\end{equation}
$[\hn\Lie\nh g](w,u)\,$ being treated here, formally, as a $\,1$-form 
sending any vector field $\,v$ to the function
$\,[\hn\Lie\hskip-1.2pt_v\w g](w,u)$. The
word `formally' reflects the fact that $\,[\hn\Lie\nh g](w,u)\,$ is not
tensorial in $\,v$.
Nevertheless, in Sect.\,\ref{tc} we point out that $\,N'$ and
$\,N''$ are well-defined tensors, and prove the following result.
\setcounter{theorem}{6}
\begin{theorem}\label{twoni}The vanishing of both\/ $\,N'\nnh$ and\/
 $\,N''\nnh$
is necessary and sufficient for the in\-te\-gra\-bi\-li\-ty of the given
symmetric\/ $\,(0,2)\,$ tensor\/ $\,g\,$ of constant rank $\,r$.
\end{theorem}
The lines following formula (\ref{bcm}) in Sect.\,\ref{ik} provide a set 
of $\,d_1\w$ Nijen\-huis-type tensors for an
algebraically constant nil\-po\-tent $\,(1,1)\,$ tensor $\,\jt\,$ with
the Jor\-dan normal form $\,d_1\w\nnh\ldots d_m\w$. This set consists of
$\,N\hs$ -- see (\ref{nih}) -- and $\,d_1\w\nh-1$ additional tensors
responsible for in\-te\-gra\-bi\-li\-ty of $\,\mathrm{Ker}\,\jt^i$, 
$\,1\le i<d_1$ (which makes them redundant in the case (\ref{csd}), due to
Theorem~\ref{cntrl}).

Finally, formulae (\ref{hni}) -- (\ref{omg}) in Sect.\,\ref{ik} define, for
(skew)sym\-me\-tric $\,(2,0)$ tensors $\,\jt\,$ of constant rank $\,r$, a 
Nijen\-huis-type $\,(2r+3,0)\,$ tensor 
$\,\hni\hs$ such that $\,\hni\nnh=\hh0\,$ if and only if
$\,\mathrm{Im}\,\hs\jt\hs$ is in\-te\-gra\-ble. When $\,\jt\,$ is symmetric,
vanishing of $\,\hni\hs$ thus amounts, by Proposition~\ref{inttz}, to
in\-te\-gra\-bi\-li\-ty of $\,\jt$. However, for a rank $\,r$ 
{\it bi\-vec\-tor\/} $\,\jt$, the condition $\,\hni\nnh=\hh0$, despite still
being necessary, is not sufficient in order that $\,\jt\,$ be
in\-te\-gra\-ble. Obvious examples illustrating the last claim arise, cf.\ 
Proposition~\ref{crbiv}, on a product manifold
$\,M\nh=\vs\nnh\times\nnh\vs'\nh$, with $\,\jt\,$ obtained as the trivial
extension to $\,M\,$ of the inverse of a non\-clos\-ed nondegenerate
$\,2$-form on $\,\vs$.

\renewcommand{\thetheorem}{\thesection.\arabic{theorem}}
\renewcommand{\theprop}{\thesection.\arabic{prop}}
\section{Preliminaries}\label{pr}
\setcounter{equation}{0}
Manifolds (by definition connected) and mappings, including sections of
bundles, are always assumed to be smooth. Tensor fields will usually be
referred to as {\it tensors}. All vector spaces are 
real (except in Sect.\,\ref{cd}) and fi\-nite-di\-men\-sion\-al. 

Given a manifold $\,M\,$ and vector sub\-bun\-dles $\,\dz,\dz'\nh,\dz''$ 
of $\,T\nh M\nh$, we write
\begin{equation}\label{lbr}
[\dz,\dz']\,\subseteq\,\dz'' 
\end{equation}
when $\,[w,w']\,$ is a local section of $\,\dz''$ for any local sections
$\,w\,$ of $\,\dz\,$ and $\,w'$ of $\,\dz'\nh$.
\begin{lemma}\label{cdone}For\/ $\,M\nh,\dz,\dz'$ as above, suppose that\/ 
$\,\dz\hs$ contains $\,\dz'\nnh$ with co\-dimen\-sion one, and\/
$\,[\dz,\dz']\subseteq\dz$. Then\/ $\,\dz\,$ is in\-te\-gra\-ble.
\end{lemma}
\begin{proof}As $\,[\dz'\nh,\dz']\subseteq\dz$, the relation
$\,[\dz,\dz]\subseteq\dz\,$ follows if we note that, locally, sections of
$\,\dz\hs$ have the form $\,v+\phi w\,$ for various sections $\,v\,$ of
$\,\dz'\nh$, functions $\,\phi$, and one fixed section $\,w\,$ of $\,\dz$.
\end{proof}
Let $\,\pi:M\nh\to\vs\,$ be a mapping between manifolds. We say that
a vector field $\,w\,$ (or, a distribution $\,\zz$) on $\,M\,$ is 
$\,\pi\hn${\it-pro\-ject\-a\-ble\/} if $\,d\pi\nnh_x\w w_x\w=\hs u_{\pi(x)}\w$ 
or, respectively,
$\,d\pi\nnh_x\w(\zz\hn_x\w)\,=\,\wz\nnh_{\pi(x)}\w$ for all
$\,x\in M\,$ and some vector field $\,u\,$ (or, distribution $\,\wz$)
on $\,\vs$. If this is the case,
\begin{equation}\label{pri}
\mathrm{the\ in\-te\-gra\-bi\-li\-ty\ of\ }\,\zz\,\mathrm{\ implies\ that\ of\
}\,\wz\nh,
\end{equation}
since $\,\pi\,$ restricted to any leaf of $\,\zz\,$ is, locally, a
submersion onto an integral manifold of $\,\wz\nnh$. We also define 
$\,\pi$-pro\-ject\-a\-bil\-i\-ty of a $\,(0,q)\,$ tensor field $\,\jt\hn\,$
on $\,M\nh\,$ by requiring $\,\jt\,$ to be the $\,\pi$-pull\-back of a 
$\,(0,q)\,$ tensor field on $\,\vs$.

Given an integrable distribution
$\,\vz\,$ 
on a manifold $\,M\nh$, 
every point of $\,M\,$ has a neighborhood $\,\,U$ such that, for some 
manifold $\,\vs$, the leaves of $\,\vz\,$ restricted to $\,\,U\hs$ are the
fibres of a bundle projection $\,\pi\nnh:\nh U\to\vs$.

Let $\,\vz\,$ be an in\-te\-gra\-ble distribution 
on a manifold $\,M\nh$. By $\,\vz\nh${\it-pro\-ject\-a\-bil\-i\-ty\/} of a
vector field on 
an open set $\,\,U'\subseteq M\,$ (or, of a distribution on $\,\,U'\nh$, or of
a $\,(0,q)$ tensor field on $\,\,U'$) we mean
its $\,\pi$-pro\-ject\-a\-bil\-i\-ty for any $\,\pi,\hs U\nh,\vs\,$ as in the
last paragraph such that $\,\,U\subseteq U'\nh$. Then, for a vector field
$\,w\,$ on $\,M\nh$,
\begin{equation}\label{prj}
\begin{array}{l}
w\,\,\mathrm{\ is\ }\,\,\vz\hyp\mathrm{pro\-ject\-a\-ble\ if\ and\ only\ if,\ 
for\ every\ section}\\
v\,\mathrm{\ of\ }\hs\vz\nh\mathrm{,\ the\ Lie\ 
bracket\ }\hs[v,\nh w]\hs\mathrm{\ is\ also\ a\ section\ of\
}\hs\vz\nh.
\end{array}
\end{equation}
(This is obvious in local 
coordinates for $\,M\,$ turning $\,\pi\hs$ as above into a 
Car\-te\-sian-prod\-uct projection.) It is also clear that, given a
$\,(0,q)\,$ tensor field $\,\jt$,
\begin{equation}\label{prt}
\begin{array}{l}
\jt\hs\mathrm{\ is\ }\hs\vz\hyp\mathrm{pro\-ject\-a\-ble\ if\ and\ only\
if\ }\nh\,d_v\w[\jt(w_1,\dots,w_q\w)]\hh=\hh\hs0\nh\,\mathrm{\ for\ all\ sec}\hyp\\
\mathrm{tions\ }\hs\,v\hs\,\mathrm{\ of\ }\,\hs\vz\,\mathrm{\ and\ all\
}\,\hs\vz\hyp\mathrm{pro\-ject\-a\-ble\ local\ vector\ fields\
}\,\,w_1,\dots,w_q\w\hh.
\end{array}
\end{equation}
\begin{lemma}\label{prdis}For an in\-te\-gra\-ble distribution\/ $\,\vz\,$
and any distribution\/ $\,\zz\,$ on an $\,n$-di\-men\-sion\-al manifold\/
$\,M\nh$, the following two
conditions are equivalent.
\begin{enumerate}
\item[{\rm(a)}] $\zz\,$ is\/ $\,\vz$-pro\-ject\-a\-ble,
\item[{\rm(b)}] $\zz\cap\hs\vz\hs$ has a constant dimension and\/
$\,[\vz,\zz]\subseteq\vz+\nh\zz$.
\end{enumerate}
Under the additional assumption that $\,\vz\subseteq\zz$,
\begin{enumerate}
\item[{\rm(c)}] $\zz\,$ is\/ $\,\vz$-pro\-ject\-a\-ble if and only if\/ 
$\,[\vz,\zz]\subseteq\zz$.
\end{enumerate}
\end{lemma}
\begin{proof}The equivalence of (a) and (b), once established, trivially
implies (c) when $\,\vz\subseteq\zz$. We proceed, however, by first proving
(c), in the case where $\,\vz\subseteq\zz$. It will then clearly follow from
(c) that (a) and (b) are equivalent, since $\,\vz$-pro\-ject\-a\-bil\-i\-ty
of $\,\zz\,$ amounts to $\,\vz$-pro\-ject\-a\-bil\-i\-ty of
$\,\vz+\nh\zz$, and $\,\vz\subseteq\vz+\nh\zz$.

If $\,\vz\subseteq\zz\,$ and $\,\zz\,$ is $\,\vz$-pro\-ject\-a\-ble, onto
some distribution $\,\wz\,$ on a local leaf space of $\,\vz$, then $\,\zz\,$
is spanned by $\,\vz$-pro\-ject\-a\-ble sections obtained as lifts of
sections of $\,\wz\,$ (including sections of $\,\vz\nh$, which are lifts
of $\,0$). As any section of $\,\zz$ is a functional combination of 
$\,\vz$-pro\-ject\-a\-ble ones, (\ref{prj}) yields
$\,[\vz,\zz]\subseteq\zz$. 

Conversely, let $\,\vz\subseteq\zz\,$ and $\,[\vz,\zz]\subseteq\zz$. We choose 
local coordinates $\,x^1\nh,\dots,x^n$ such that $\,\vz\,$ is spanned by the
coordinate vector fields $\,\partial\nh_i\w$, $\,i=1,\dots,m$, and a local
trivialization of the sub\-bun\-dle $\,\zz\,$ of $\,T\nh M\,$ having the form
$\,\partial\nh_1\w,\dots,\partial\nh_m\w,w_{m+1}\w,\dots,w_s\w$. Using the
index ranges $\,1\le i,j,k\le m<a,b,c\le s$, we obtain, since 
$\,[\vz,\zz]\subseteq\zz$,
\begin{equation}\label{diw}
[\partial\nh_i\w,w_a\w]\,=\,\vg_{\hskip-2.7ptia}^{\hs j}\partial\nh_j\w\hs
+\,\vg_{\hskip-2.7ptia}^{\hs b}w_b\w
\end{equation}
for some functions
$\,\vg_{\hskip-2.7ptia}^{\hs j},\,\vg_{\hskip-2.7ptia}^{\hs b}$. The
$\,w_b\w$-com\-po\-nent of the
Ja\-co\-bi identity
$\,[\partial\nh_i\w,[\partial\nh_j\w,w_a\w]]
=[\partial\nh_j\w,[\partial\nh_i\w,w_a\w]]$, with
$\,[\partial\nh_i\w,\partial\nh_j\w]=0$, now implies
{\it symmetry of\/ $\,\partial\nh_i\w\vg_{\hskip-2.7ptja}^{\hs b}
+\vg_{\hskip-2.7ptic}^{\hs b}\vg_{\hskip-2.7ptja}^{\hs c}$ in\/} $\,i,j$.
This symmetry amounts to the vanishing of the curvature,
that is, flatness, for the linear connection with the components 
$\,\vg_{\hskip-2.7ptia}^{\hs b}$ in a rank $\,s-m\,$ vector bundle over
a manifold with the coordinates $\,x^i$, $\,i=1,\dots,m$. The equations
$\,\partial\nh_i\w\psi^b\nnh+\vg_{\hskip-2.7ptic}^{\hs b}\psi^c\nh=0$,
stating
that $\,\psi^a\nh$, with $\,m<a\le s$, are the components of a parallel
section $\,\psi$, is thus locally solvable with any prescibed initial value at
a given point $\,z$. Let us choose such a section $\,\psi\nnh_a\w$, for 
$\,a=m+1,\dots,s$, with the initial 
value $\,(\delta_a^1,\dots,\delta_a^m)\,$ at $\,z$, so that 
\begin{equation}\label{dps}
\partial\nh_i\w\psi_{\!a}^b\hs
+\,\vg_{\hskip-2.7ptic}^{\hs b}\psi_{\!a}^c\hs=\,0\hh,\quad
\psi_{\!a}^b\nh(z)=\delta_a^b\hh.
\end{equation}
Setting $\,u_a\w=\psi_{\!a}^bw_b\w$, we obtain a new local
trivialization $\,\partial\nh_1\w,\dots,\partial\nh_m\w,u_{m+1}\w,\dots,u_s\w$ 
of $\,\zz\,$ while, by (\ref{diw}) and (\ref{dps}),
$\,[\partial\nh_i\w,u_a\w]\,$ are sections of $\,\zz$. Therefore, due to 
(\ref{prj}), our new local trivialization of $\,\zz\,$ consists of
$\,\vz$-pro\-ject\-a\-ble sections, which makes $\,\zz\,$ itself
$\,\vz$-pro\-ject\-a\-ble.
\end{proof}
Given a symmetric or skew-sym\-me\-tric $\,(2,0)\,$ tensor $\,\jt\hs$ in a
vector space $\,V\nh$, let $\,\mathrm{Ker}\,\jt\hs$ and 
$\,\mathrm{Im}\,\jt\hs$ be the kernel and image of
$\,V^*\nh\ni\xi\mapsto\jt(\xi,\,\cdot\,)\in V\nh$. By the {\it 
restriction\/} of $\,\jt\hs$ to $\,W=\mathrm{Im}\,\hs\jt\hs$ we mean the 
$\,(2,0)\,$ tensor $\,\jt\nh_W\w$ in $\,W\hs$ given by
\begin{equation}\label{rst}
\begin{array}{l}
\jt\nh_W\w\hn(\eta,\eta')=\jt(\xi,\xi')\,\mathrm{\ for\ any\ 
}\,\eta,\eta'\hn\in W^*\,\mathrm{and}\\
\mathrm{any\ extensions\ }\,\,\xi,\xi':V\to\bbR\,\mathrm{\ of\ 
}\,\eta,\eta'\mathrm{\ to\ }\,V\nnh.
\end{array}
\end{equation}
As $\,\xi\,$ and $\,\xi'$ are unique up to adding elements of
$\,W'\nh=\mathrm{Ker}\,\jt\subseteq V^*\nh$, the polar space of $\,W\nh$,
the restriction is well defined. In other words, since
$\,W'\nh=\mathrm{Ker}\,\jt$, the bi\-lin\-e\-ar
form $\,\jt\hs$ on $\,V^*$ descends to one on $\,V^*\nnh\nh/W'\nh=W^*\nh$,
which is our $\,\jt\nh_W\w$. (The natural identification of
$\,W^*$ with $\,V^*\nnh\nh/W'$ sends $\,\eta\in W^*$ to the
$\,W'\nh$-co\-set of any extension of $\,\eta\,$ to $\,V\nh$.) In addition,
$\,\jt\,$ is the image of $\,\jt\nh_W\w$ under the linear operator
$\,W^{\odot2}\nh\to V^{\odot2}$ or $\,W^{\wedge2}\nh\to V^{\wedge2}$ induced
by the inclusion $\,W\nh\to V\nh$. Finally,
\begin{equation}\label{ndg}
\mathrm{the\ restriction\ }\,\,\jt\nh_W\w\mathrm{\ is\ nondegenerate,}
\end{equation}
as any
$\,\eta\in W\smallsetminus\{0\}\,$ has an extension $\hs\xi$ to $\,V\hs$ not
lying in $\,W'\nh=\mathrm{Ker}\,\jt$, and hence $\,\jt(\xi,\xi')\ne0\,$ for
some $\,\xi'\nh\in V^*\nh$.
\begin{remark}\label{sbpsr}Con\-stant-rank symmetric $\,(2,0)\,$ tensors
$\,\jt\,$ on a manifold $\,M$ are naturally identified with \hbox{sub\hh-\hn}pseu\-\hbox{do\hs-}Riem\-ann\-i\-an metrics on $\,M\nh$, that is,
pseu\-\hbox{do\hs-}Riem\-ann\-i\-an fibre metrics $\,h\,$ on vector
sub\-bun\-dles $\,\bz\,$ of $\,T\nh M\nh$. In fact, one may set
$\,\bz=\mathrm{Im}\,\jt\,$ and, using (\ref{rst}) -- (\ref{ndg}), declare
$\,h\,$ to be the inverse of the restriction of $\,\jt\,$ to $\,\bz$.
Thus, cf.\ the lines preceding (\ref{ndg}), $\,\jt\,$ is the image of
the inverse of $\,h$ under the 
bun\-dle mor\-phism 
$\,\bz^{\odot2}\nh\to[T\nh M]^{\odot2}$ induced 
by the inclusion $\,\bz\to T\nh M\nh$.
\end{remark}
The $\,(r\nh-1)$-fold contraction of two $\,(r,0)$-ten\-sors 
$\,\jt,\Pi\,$ on a manifold with a fixed Riemannian metric $\,g$, appearing in 
(i) below, is
\begin{equation}\label{rmo}
\mathrm{the\ }\,(2,0)\,\mathrm{\ tensor\ }\,\beta\,\mathrm{\ given\ by\
}\,\beta\hh^{ij}\nnh
=\jt^{ii_2\ldots i_r}\nh\Pi^{jj_2\ldots j_r}\nh g_{i_2\hn j_2}\w\ldots
g_{i_r\nh j_r}\w.
\end{equation}
\begin{remark}\label{hodge}Let $\,V$ be a Euclidean $\,n$-space with the
inner product $\,\lr$.
\begin{enumerate}
\item[{\rm(i)}] The $\,(r\nh-1)$-fold contraction (\ref{rmo}) against itself
of a nonzero de\-com\-pos\-a\-ble $\,r$-vec\-tor 
$\,v_1\w\wedge\ldots\wedge v_r\w\in V\nh^{\wedge r}$ yields a
$\,(2,0)\,$ tensor which, viewed with the aid of $\,\lr\,$ as an
en\-do\-mor\-phism of $\,V\nh$, equals a nonzero multiple of the
or\-thog\-onal 
projection onto the span of $\,v_1\w,\dots,v_r\w$. (To see this, we are free
to assume that $\,v_1\w,\dots,v_r\w$ are or\-tho\-nor\-mal.)
\item[{\rm(ii)}] If $\,V\nh$ is oriented,
$\,*(e_1\w\wedge\ldots\wedge e_r\w)=e_{r+1}\w\wedge\ldots\wedge e_n\w$
for the Hodge star 
$\,*:V\nh^{\wedge r}\nh\to V\nh^{\wedge(n-r)}$ and 
any positive or\-tho\-nor\-mal basis $\,e_1\w,\dots,e_n\w$ of $\,V\nh$.
\end{enumerate}
\end{remark}
\begin{remark}\label{nwcon}In an $\,s\times(n-s)\,$ product
$\,n$-di\-men\-sion\-al manifold
$\,M\,$ with global product coordinates $\,x^i\nh,x^a$ (index ranges 
$\,1\le i\le s<a\le n$), let the component functions 
$\,g_{ij}\w,\jt^{ij}$ and $\,\vg_{\hskip-2.7ptij}^{\hs k}=
\vg_{\hskip-2.7ptji}^{\hs k}$ represent families of $\,(0,2)\,$ tensors,
$\,(2,0)\,$ tensors and tor\-sion-free
connections on the leaves of the in\-te\-gra\-ble distribution
spanned by the coordinate vector fields $\,\partial\nh_i\w$. Suppose that
each tensor is parallel relative to the corresponding connection on the leaf: 
$\,\partial\nh_i\w g_{jk}\w=\vg_{\hskip-2.7ptij}^{\hs l}g_{lk}\w
+\vg_{\hskip-2.7ptik}^{\hs l}g_{jl}\w$ and 
$\,\partial\nh_i\w\jt^{jk}=-\vg_{\hskip-2.7ptil}^{\hs j}\jt^{lk}\nh
-\vg_{\hskip-2.7ptil}^{\hs k}\jt^{jl}\nh$. Setting
$\,g_{\lambda\mu}\w=\jt^{\lambda\mu}\nh
=\vg_{\hskip-2.7pt\lambda\mu}^{\hs\nu}=0\,$ whenever at least one of the
indices $\,\lambda,\mu,\nu\in\{1,\dots,n\}\,$ is in the $\,a\,$ range, we
extend the above data to their analogs defined on $\,M\nh$,
namely, a $\,(0,2)\,$ tensor $\,g$, a $\,(2,0)\,$ tensor $\,\jt\,$ and a
tor\-sion-free connection $\,\nabla\nh$, in such a way that, obviously,
$\,\nabla\nh g=\nabla\nh\jt=0$.
\end{remark}

\section{The com\-plex-di\-ag\-o\-nal\-iz\-able case}\label{cd}
\setcounter{equation}{0}
By (\ref{nih}), for a $\,(1,1)\,$ tensor field $\,\jt\,$ and any $\,a\in\bbR$,
\begin{equation}\label{sni}
\jt\,\mathrm{\ \ \,and\ \ }\,\jt\hs\,-\,a\hs\mathrm{Id}\,\mathrm{\ \ have\
the\ same\ Nijen\-huis\ tensor.}
\end{equation}
To justify Remark~\ref{cdiag}, we invoke a result of
Bol\-si\-nov, Ko\-nya\-ev and Mat\-ve\-ev
\cite[Theorem 3.2]{bolsinov-konyaev-matveev}. It states that, if a 
$\,(1,1)\,$ tensor $\,\jt\,$ with $\,N\nnh=\hh0\,$ on a manifold $\,M$ has
complex characteristic roots of constant (algebraic) multiplicities, then
$\,M\,$ and $\,\jt\,$ are, locally, decomposed into Cartesian products 
of factor man\-i\-folds/ten\-sors with $\,N\nnh=\hh0$, where each factor
corresponds to (and realizes) a real eigen\-value function of $\,\jt$, or a
conjugate pair of its (non\-real) complex char\-ac\-ter\-is\-tic-root
functions.

Under the assumption made in Remark~\ref{cdiag}, the complex characteristic
roots of $\,\jt\,$ are all constant. Let the symbols $\,M\,$ and $\,\jt\,$
now stand for one of of factor man\-i\-folds/ten\-sors with $\,N\nnh=\hh0$,
mentioned above.

If the (constant) eigen\-value realized by this $\,\jt\,$ is real, our claim
follows: $\,\jt$ equals a constant multiple of $\,\mathrm{Id}$.

Otherwise, the characteristic roots realized by $\,\jt\,$ 
are $\,a\pm b\hh i$, with $\,a,b\in\bbR\,$ and $\,b\ne0$. Let
$\,J=b^{-\nnh1}\nh(\jt-a\hs\mathrm{Id})$. By (\ref{sni}), $\,J\,$ still has
$\,N\nnh=\hh0$, while the characteristic roots of $\,J\,$ (and hence those of
its complexification $\,\hat J\hs$) are $\,i\,$ and $\,-i$. As $\,\hat J\,$ is
di\-ag\-o\-nal\-iz\-able -- due to our assumption -- we get 
$\,\hat J\hh^2\nh=-\mathrm{Id}$. Thus, $\,J\hh^2\nh=-\mathrm{Id}$, and
the local constancy of $\,\jt\,$ follows from the New\-land\-er-Ni\-ren\-berg
theorem.

The more modest goal of establishing a weaker version of
Remark~\ref{cdiag}, with in\-te\-gra\-bi\-li\-ty of $\,\jt\,$ replacing its
local constancy, is easily achieved as follows. Rather than invoking the
New\-land\-er-Ni\-ren\-berg theorem, one shows that an al\-most-com\-plex 
structure $\,J\,$ or, more generally, a $\,(1,1)$ tensor $\,J\,$ with
$\,N\nnh=\hh0\,$ and $\,J\hh^2\nh=c\hs\mathrm{Id}$, where
$\,c\in\bbR\smallsetminus\{0\}$, has $\,\hat\nabla\nnh J=0\,$ 
for some tor\-sion-free connection $\,\hat\nabla\nh$.

We start from 
any tor\-sion-free connection $\,\nabla\nh$. By (\ref{tfr}),
$\,4\hh c\hh[B_v\w w-B_w\w v]=N(v,w)\,$ 
for $\,B_v\w w\,$ given by
$\,4\hh cB_v\w w=2J[\nabla_{\nnh\!v}\w\nnh J]w
+J[\nabla_{\nnh\!w}\w\nh J]v+[\nabla_{\!\!J\hn w}\w J]v$, that is, 
$\,4\hh cB_v\w w=(J[\nabla_{\nnh\!v}\w\nnh J]w
+J[\nabla_{\nnh\!w}\w\nh J]v)
+(J[\nabla_{\nnh\!v}\w\nnh J]w+[\nabla_{\!\!J\hn w}\w J]v)$. Thus, 
the vanishing of $\,N$ for $\,J\,$ amounts to symmetry of $\,B_v\w w\,$ 
in $\,v,w$, while $\,[J,B_v\w]=\nabla_{\nnh\!v}\w\nnh J\,$ since, 
$\,J^2$ being parallel, $\,\nabla_{\nnh\!v}\w\nnh J\,$ anti\-com\-mutes with 
$\,J$. This is precisely the relation $\,\hat\nabla\nnh J=0$ for the
tor\-sion-free connection $\,\hat\nabla\hs$ characterized by 
$\,\hat\nabla_{\nnh\!v}\w=\nabla_{\nnh\!v}\w\nnh+B_v\w$.

The assignment $\,\nabla\mapsto\hat\nabla\nh=\nabla\hn+B\,$ 
appearing above is a natural projection of the af\-fine space of all
tor\-sion-free connections on the 
manifold in question onto
the af\-fine sub\-space formed by those connections which make $\,J\,$
parallel.

The above conclusion is due to Clark and Bruck\-heim\-er 
\cite[Theorem 6]{clark-bruckheimer}. Our argument is a concise version of
one used, in a more general
situation, by Hernando, Reyes and Gadea 
\cite[Theorems 3.4 and 7.1]{hernando-reyes-gadea}.


\section{Tensors of type $\,(1,1)$}\label{tt}
\setcounter{equation}{0}
For the reader's convenience, we repeat here the definition,
due to Nijen\-huis \cite{nijenhuis}, of the Nijen\-huis tensor (\ref{nih})
associated with a $\,(1,1)$ tensor $\,\jt\,$ on a manifold:
\begin{equation}\label{nij}
N(v,w)\,=\,\jt[\jt v,w]\,+\,\jt[v,\jt w]\,
-\,[\jt v,\jt w]\,-\,\jt^2\hn[v,w]\hh.
\end{equation}
Applying $\,\jt^i$ to both sides, with any integer $\,i\ge0$, one obviously
obtains
\begin{equation}\label{thi}
\jt^{i+1}\hn(\jt[v,w]-[v,\jt w])\,
=\,\jt^i\hn(\jt[\jt v,w]-[\jt v,\jt w])\,-\,\jt^i[N(v,w)]\hh.
\end{equation}
Let $\,N\nnh=\hh0$. For any vector fields $\,v,w\,$ and integers $\,i,j\ge0$,
\begin{equation}\label{ift}
\mathrm{if\ }\,\jt^i\nh v=0\mathrm{,\ then\ }\,\jt^i\mathrm{\ also\
annihilates\ }\,\jt^j\nh[v,w]-[v,\jt^j\nh w]\hh.
\end{equation}
Namely, let $\,R(i,j)\,$ be the assertion (\ref{ift}), and $\,R(j)\,$ the
claim that $\,R(i,j)\,$ holds for all $\,i\ge1$. Now $\,R(1,1)\,$ 
is immediate from (\ref{nij}), while, assuming $\,R(i,1)$, and choosing any
$\,v\,$ with $\,0=\jt^{i+1}\nh v=\jt^i\nh\jt v$, we get, from
$\,R(i,1)\,$ for $\,\jt v\,$ (not $v$), zero on the right-hand side of
(\ref{thi}), and hence also on the left-hand side, which yields $\,R(i+1,1)\,$ 
and, by induction on $\,i$, establishes $\,R(i,1)\,$ for all $\,i\ge1$, that
is $\,R(1)$. If we now assume $\,R(j)$, and use any $\,i\ge1$, we see that 
$\,\jt^i\nh[v,\jt^{j+1}\nh w]=\jt^{i+1}\nh[v,\jt^j\nh w]\,$ when
$\,\jt^i\nh v=0\,$ (from $\,R(i,1)\,$
applied to $\,\jt^j\nh w\,$ rather than $\,w$), which in turn equals 
$\,\jt^{i+j+1}\nh[v,w]\,$ (due to $\,R(i+1,j)$, a consequence of $\,R(j)$). One
thus has $\,R(j+1)$, which completes the proof of (\ref{ift}).

When, again, $\,N\nnh=\hh0\,$ in (\ref{nij}) and $\,i,j,k\,$ are nonnegative
integers,
\begin{equation}\label{ziz}
\begin{array}{l}
\mathrm{a)}\hskip6pt[\bz\hh^i\nh,\bz\hh^i]\subseteq\bz\hh^i\nh,
\quad\mathrm{b)}\hskip6pt[\zz^i\nh,\bz^j]
\subseteq\zz^i\hskip-1.1pt+\hs\bz^j\nh,\quad
\mathrm{c)}\hskip6pt[\zz^i\nh,\zz^j]\subseteq\zz^{i+j}\nh,\\
\mathrm{d)}\hskip6pt[\zz^i\nh,\zz\hn^k]\subseteq\zz\hn^k\mathrm{\ \ if\
}\,\zz^i\mathrm{\ is\ in\-te\-gra\-ble\ and\ }\,k\ge i
\end{array}
\end{equation}
-- notation of (\ref{lbr}) -- with $\,\jt\,$ assumed algebraically constant. In
fact, (\ref{ziz}\hh-a), that is,
the in\-te\-gra\-bi\-li\-ty of each $\,\bz\hh^i\nh$,
follows via induction on $\,i$, from (\ref{nij}) with $\,N\nnh=\hh0\,$ and
with $\,v,w\,$ replaced by $\,\jt^i v,\jt^i w$. (The third Lie bracket in
(\ref{nij}) then is a section of $\,\bz\hh^{i+1}\nh$, once we assume that
$\,[\bz\hh^i\nh,\bz\hh^i]\subseteq\bz\hh^i\nh$.) For (\ref{ziz}\hh-b), note 
that the Lie bracket of sections $\,v\,$ of $\,\zz^i$ and $\,\jt^j\nh w\,$ 
of $\,\bz^j$ equals, by (\ref{ift}), $\,\jt^j\nh[v,w]$ plus a section of 
$\,\zz^i$. Finally, (\ref{ziz}\hh-c) and (\ref{ziz}\hh-d) are further
consequences of (\ref{ift}): given sections $\,v\,$ of $\,\zz^i$ and $\,w\,$ of
$\,\zz^j\nh$, (\ref{ift}) reads $\,\jt^{i+j}\nh[v,w]=0$, while (\ref{ift}) 
with $\,j=k-i$, for sections $\,v\,$ of $\,\zz^i$ and $\,w\,$ of
$\,\zz\hn^k$ 
(which makes $\,\jt^j\nh w\,$ and $\,[v,\jt^j\nh w]\,$ sections of $\,\zz^i$
-- the latter due to the assumed in\-te\-gra\-bi\-li\-ty of $\,\zz^i$),
yields $\,\jt^k\nh[v,w]=0$.

The conclusion (\ref{ziz}\hh-a) is due to Bol\-si\-nov, Ko\-nya\-ev and
Mat\-ve\-ev \cite[Cor.\,2.5.]{bolsinov-konyaev-matveev}.

\section{Proof of Theorem~\ref{npjor}}\label{io}
\setcounter{equation}{0}
We use induction on the dimension, with the following induction step.
\begin{lemma}\label{indst}Let\/ $\,\jt\hs$ be an algebraically constant\/
$\,(1,1)\,$ tensor with\/ $\,N\nnh=\hh0\,$ in\/ {\rm(\ref{nih})} on a
manifold\/ $\,M\hs$ 
such that the distribution\/ 
$\,\zz=\mathrm{Ker}\,\jt\hs$ is in\-te\-gra\-ble.
\begin{enumerate}
\item[(a)] $\jt$-im\-ages of\/ $\,\zz$-pro\-ject\-a\-ble local vector fields
in\/ $\,M\hs$ are themselves\/ $\,\zz$-pro\-ject\-a\-ble, so that\/ 
$\,\jt\hs$ naturally descends to a\/ $\,(1,1)\,$ tensor\/ $\,\hjt\,$ on 
any local leaf space\/ $\,\vs\nh\,$ of\/
$\,\zz$, and\/ $\,\hjt\nh\,$ also has\/ $\,N\nnh=\hh0$.
\item[(b)] If local vector fields\/ $\,v,w$, and hence also\/ $\,\jt v,\jt w$, 
are\/ $\,\zz$-pro\-ject\-a\-ble and the projected images of\/
$\,v,w,\jt v,\jt w\,$ all commute, then $\,[\jt v,\jt w]=0$.
\item[(c)] Nil\-po\-tency of\/ $\,\jt$, or in\-te\-gra\-bi\-li\-ty of the
distributions\/ $\,\mathrm{Ker}\,\jt^i$ for all\/ $\,i\ge1$,
implies the same property for\/ $\,\hjt$.
\end{enumerate}
\end{lemma}
\begin{proof}Applying (\ref{nih}) to $\,v\,$ with $\,\jt v=0\,$ and $\,w\,$
pro\-ject\-a\-ble along $\,\zz$, we obtain $\,\jt[v,\jt w]=0$, as 
$\,\jt[v,w]$, and hence $\,\jt^2\hn[v,w]$, vanishes due to
pro\-ject\-a\-bil\-i\-ty of $\,w\,$ and (\ref{prj}). By (\ref{prj}), this
proves the first part of (a), 
with an obvious definition of $\,\hjt$. Evaluating (\ref{nih}) on
pro\-ject\-a\-ble vector fields, or applying $\,\jt\,$ to them, we get
$\,N\nnh=\hh0\,$ for $\,\hjt\,$ or, respectively, the claim about
nil\-po\-tency in (c).

Under the assumptions of (b), $\,[\jt v,w],[v,\jt w]\,$ and $\,\jt[v,w]\,$
are -- by (a) -- pro\-ject\-a\-ble onto $\,0$, which makes them sections
of $\,\zz=\mathrm{Ker}\,\jt$, so that (\ref{nih}) with
$\,N\nnh=\hh0\,$ yields $\,[\jt v,\jt w]=\jt([\jt v,w]+[v,\jt w]-\jt[v,w])
=0$.

If the distributions $\,\zz^i\nh=\mathrm{Ker}\,\jt^i$ are all
in\-te\-gra\-ble, pro\-ject\-a\-ble vector fields that project onto sections
of $\,\mathrm{Ker}\,\hs\hjt\hs^i$ span the distribution
$\,\zz^{i+1}$ (the $\,\jt^i\nh$-pre\-im\-age of the vertical
distribution $\,\zz$). Pro\-ject\-a\-bil\-i\-ty of each $\,\zz^{i+1}\nh$,
immediate from that of $\,\jt$, or from (\ref{ziz}\hh-d) and 
Lemma~\ref{prdis}(c), combined with (\ref{pri}), proves (c).
\end{proof}
The assertion $\,N\nnh=\hh0\,$ in (a) is also  a special case of 
\cite[Prop.\,2.4]{bolsinov-konyaev-matveev}.
\begin{proof}[Proof of Theorem~\ref{npjor}]As the implications \hbox{(ii)
$\implies$ (iii) $\implies$ (iv)
$\implies$ (i)} are obvious -- the last two from from (\ref{imp}), (\ref{iii})
and (\ref{tfr}) -- we now just proceed to show that (ii) holds whenever (i)
does,
using induction on $\,n\ge1$. The case $\,n=1\,$ being trivial,
we now fix $\,n>1$ and assume that (i) implies (ii) in dimensions less
than $\,n$, while (i) is satisfied on an $\,n$-man\-i\-fold $\,M\nh$, with 
$\,\jt\ne0$. Using $\,\hjt\,$ and a 
local leaf space $\,\vs\,$ arising from Lemma~\ref{indst}(a), and replacing 
$\,M\,$ by a suitable neighborhood of a given point, we get a 
bundle projection $\,\pi:M\to\vs\,$ with the vertical distribution
$\,\zz=\mathrm{Ker}\,\jt$, while (i), and hence (ii), holds for $\,\hjt$, on
$\,\vs$, since $\,\dim\vs<n$. The resulting commuting Jor\-dan-form frame
for $\,\hjt\,$ is split into $\,\hjt${\it-or\-bits\/} 
$\,u_1\w,\dots,u\hn_d\w$ of various {\it lengths\/} $\,d\ge1$, with the {\it
initial\/} vector (field) $\,u_1\w$ lying in $\,\mathrm{Ker}\,\hs\hjt$,
the {\it final\/} vector $\,u\hn_d\w$ outside of $\,\mathrm{Im}\,\hs\hjt$, and
$\,u_i\w=\hjt\hs^{d-i}\hs u\hn_d\w$ for $\,i=1,\dots,d$.

We now associate with every given 
$\,\hjt\hs$-or\-bit $\,u_1\w,\dots,u\hn_d\w$ the corresponding $\,\jt$-or\-bit
$\,v_0\w,v_1\w,\dots,v\nh_d\w$ of length $\,d+1\,$ in $\,M\nh$. First, we
choose 
each final vector field $\,v\nh_d\w$, on $\,M\nh$, so that it projects onto 
$\,u\hn_d\w$ under $\,\pi$, and set $\,v_i\w=\jt^{d-i}\nh v\nh_d\w$, 
$\,i=0,\dots,d-1$. We call $\,v_0\w$ the {\it pre-in\-i\-tial\/} vector. 
Our $\,v\nh_d\w$ is only unique up to adding sections
of $\,\zz=\mathrm{Ker}\,\jt\,$ (and will be modified later); this is also the
obvious reason why {\it the
non\-fi\-nal vectors\/ $\,v_0\w,v_1\w,\dots,v\nh_{d-1}\w$ are uniquely
determined}. Due to $\,\pi$-pro\-ject\-a\-bil\-i\-ty 
of $\,\jt\,$ onto $\,\hjt\,$ in Lemma~\ref{indst}(a), the resulting
$\,\jt$-or\-bit, with the pre-in\-i\-tial vector removed, projects onto the
original $\,\hjt\hs$-or\-bit, while the pre-in\-i\-tial vectors are sections of
$\,\zz\cap\hs\mathrm{Im}\,\jt$, projecting onto zero. Also, by
Lemma~\ref{indst}(b), {\it the non\-fi\-nal vectors from the union of all
the\/ $\,\jt\hn$-or\-bits commute with one another\/} (which includes the
pre-in\-i\-tial ones). Denoting by $\,k\,$ the total number of these
commuting vectors, we see that they generate
\begin{equation}\label{fre}
\mathrm{a\ free\ local\ action\ of\ }\,\bbR\hn^k\mathrm{\ in\ }\,M.
\end{equation}
The union of all the $\,\jt$-or\-bits forms a linearly independent system
at every point: the non-pre-in\-i\-tial ones project onto a frame in $\,\vs$,
which makes them linearly independent over $\,\zz=\mathrm{Ker}\,\jt\,$
(meaning linear independence of their images in 
$\,T\nh M/\zz$), while the pre-in\-i\-tial ones, lying in $\,\zz$, are
linearly independent, being the $\,\jt\hh$-im\-ages of the initial vectors, 
linearly independent over $\,\zz$.

Next, we modify -- as announced above -- the final vectors $\,e_a\w$ chosen in
$\,M\nh$, and augment the union of all the $\,\jt$-or\-bits with some sections
$\,e\nh_\lambda\w$, so as to obtain a commuting frame in $\,M\,$ which,
automatically, will be a Jor\-dan-form frame for $\,\hjt$. (The indices
$\,a,\lambda\,$ have some appropriate ranges.) To this end, we identify
$\,M\nh$, locally, with a Cartesian product of a horizontal factor (our leaf
space $\,\vs$) and a vertical factor, tangent to $\,\zz$. Our $\,e_a\w$ and
$\,e\nh_\lambda\w$ are suitable systems of commuting vector fields on the
factor manifolds, trivially extended to vector fields in $\,M\,$ (which causes
$\,e_a\w$ to commute with $\,e\nh_\lambda\w$). In a first step, for $\,e_a\w$
we choose the final vectors of our Jor\-dan-form frame for $\,\hjt$, and for 
$\,e\nh_\lambda\w$ some vertical coordinate vector fields chosen, locally, 
so as to be linearly independent over $\,\zz\cap\hs\mathrm{Im}\,\jt\,$ and
represent, under the quo\-tient-bun\-dle projection, a local trivialization
of $\,\zz/(\zz\cap\hs\mathrm{Im}\,\jt)$. Let
$\,Q\,$ now be one leaf of the in\-te\-gra\-ble distribution spanned by 
all $\,e_a\w$ and $\,e\nh_\lambda\w$. Thus, $\,Q\,$ has co\-dimen\-sion $\,k\,$
in $\,M\,$ and is transverse to the orbits of the local free action
(\ref{fre}). We now modify all $\,e_a\w$ and $\,e\nh_\lambda\w$ further, by
using the action (\ref{fre}) to spread their restrictions to $\,Q\,$ from
$\,Q\,$ to a neighborhood of $\,Q\,$ in $\,M\nh$. Due to equi\-var\-i\-ance
of $\,\pi\,$ relative to the action (\ref{fre}) and the analogous free action
in $\,\vs\,$ generated by the non\-fi\-nal vectors from the union of all 
the $\,\hjt\hs$-or\-bits, and the invariance of the final vectors in $\,\vs\,$
under the latter action, the modified $\,e_a\w$ still project onto the final
vectors
(and $\,e\nh_\lambda\w$ onto $\,0$, as the action leaves $\,\zz\,$ invariant).
Finally, $\,e_a\w$ and $\,e\nh_\lambda\w$ commute both with the non\-fi\-nal
vectors from the union of the $\,\jt$-or\-bits, and with one another: the
former follows from their $\,\bbR\hn^k\nh$-in\-var\-i\-ance, the latter since
their restrictions to $\,Q\,$ commute. This completes the proof.
\end{proof}

\section{Algebraic constancy and connections}\label{ac}
\setcounter{equation}{0}
Given a real vector bundle $\,E\hs$ of rank $\,k\,$ over a manifold $\,M\,$
and integers $\,p,q\ge0$, we say that a smooth section $\,\jt\,$ of
$\,E^{\otimes p}\nh\otimes[E^*]^{\otimes q}$ is {\it algebraically constant\/}
when, for any $\,x,y\in M\nh$, some linear iso\-mor\-phism
$\,E\nh_x\w\to E\nh_y\w$ sends $\,\jt\nnh_x\w$ to $\,\jt\nnh_y\w$.
In this case, fixing $\,z\in M\,$ and an ordered basis
$\,\mathbf{e}=(e\hn_1,\dots,e\nh_k)\,$ of $\,E\nh_z\w$, let us
\begin{equation}\label{den}
\mathrm{denote\ by\ }\,\mathbf{e}\hs\jt\nnh_z\w\mathrm{\ the\ system\
of\ components\ of\ }\,\jt\nnh_z\w\mathrm{\ in\ the\ basis\
}\,\mathbf{e}\hh,   
\end{equation}
that is, the $\,(p,q)\,$ tensor in $\,\bbR\nh^k$ arising as the image of 
$\,\jt\nnh_z\w$ under the linear iso\-mor\-phism 
$\,E\nh_z\w\nh\to\bbR\nh^k$ associated with $\,\mathbf{e}$. We now 
define two objects, the first being 
the matrix group $\,G\subseteq\mathrm{GL}\hh(k,\bbR)\,$
formed by all transition matrices between $\,\mathbf{e}\,$ and all
ordered bases $\,\bar{\mathbf{e}}\,$ of $\,E\nh_z\w$ such that
$\,\bar{\mathbf{e}}\hs\jt\nnh_z\w=\mathbf{e}\hs\jt\nnh_z\w$. In
other words, $\,G\,$ is the isotropy group of
$\,\mathbf{e}\hs\jt\nnh_z\w$ for
the obvious action of $\,\mathrm{GL}\hh(k,\bbR)\,$ on $\,(p,q)$ tensors in
$\,\bbR\nh^k\nh$.

The second one, a $\,G$-prin\-ci\-pal bundle $\,P\hs$ over $\,M\nh$, is
contained in the $\,\mathrm{GL}\hh(k,\bbR)$-prin\-ci\-pal bundle $\,Q\,$ over 
$\,M\,$ naturally associated with $\,E\nh$, and the fibre of $\,P\hs$ over any
$\,x\in M\,$ consists of the ordered bases $\,\tilde{\mathbf{e}}\,$ of
$\,E\nh_x\w$ having
$\,\tilde{\mathbf{e}}\hs\jt\nnh_x\w=\mathbf{e}\hs\jt\nnh_z\w$.

Smoothness of $\,P\hs$ follows since $\,P\hs$ is the pre\-im\-age of
the point $\,\mathbf{e}\hs\jt\nnh_z\w$ under the submersion
$\,\varPhi:Q\to\varSigma\,$ sending any ordered basis $\,\hat{\mathbf{e}}\,$ of
$\,E\nh_x\w$, at any $\,x\in M\nh$, to $\,\hat{\mathbf{e}}\hs\jt\nnh_x\w$,
with $\,\varSigma\,$ denoting the $\,\mathrm{GL}\hh(k,\bbR)$-or\-bit
of $\,\mathbf{e}\hs\jt\nnh_z\w$ viewed, again, as a $\,(p,q)\,$ tensor in 
$\,\bbR\nh^k\nh$. The submersion property of $\,\varPhi\,$ is obvious: even
the restriction of $\,\varPhi\,$ to any fibre $\,Q\nh_x\w$ of $\,Q\,$ is a
submersion, dif\-feo\-mor\-phic\-al\-ly equivalent to the projection
$\,\mathrm{GL}\hh(k,\bbR)\to\mathrm{GL}\hh(k,\bbR)/G$.

Thus, a smooth section $\,\jt\,$ of
$\,E^{\otimes p}\nh\otimes[E^*]^{\otimes q}$ is parallel relative to some
linear connection
$\,\nabla\hs$ in $\,E\,$ if and only if it is algebraically constant 
\cite[Theorems\,1-2]{kurita}, the `only if' (or, `if') claim being
obvious since $\,M\,$ is assumed connected
or, respectively, since $\,\nabla$ induced by any principal $\,G$-connection
in $\,P\hs$ clearly makes $\,\jt$ parallel. Such connections are precisely
the linear connections in $\,E\hs$ characterized by vanishing of their {\it
inner torsion\/} in the sense of \cite[Sect.\,5]{piccione-tausk}.
\begin{remark}\label{gstrc}Our construction depends on the choice of
$\,z\in M\,$ and an ordered basis $\,\mathbf{e}\,$ of $\,E\nh_z\w$. However,
different choices lead to equi\-var\-i\-ant\-ly equivalent objects. The case
of importance to us is $\,E=T\nh M\nh$, where $\,P\hs$ is the
$\,G$-struc\-ture associated with the given algebraically constant
$\,(p,q)\,$ tensor $\,\jt$. When $\,(p,q)=(1,1)\,$ and $\,\jt$ is
nil\-po\-tent, {\it we will always use\/ $\,z\,$ and\/ $\,\mathbf{e}\,$ 
realizing the Jor\-dan normal form\/ $\,d_1\w\nnh\ldots d_m\w$ of\/} $\,\jt$, 
defined as in (\ref{wdi}).
\end{remark}

\section{The Lie brackets of a local Jor\-dan frame}\label{lb}
\setcounter{equation}{0}
Recall our convention (\ref{wdi}) about representing the Jor\-dan normal
forms of nil\-po\-tent $\,(1,1)\,$ tensors in dimension $\,n\,$ as weakly
decreasing strings $\,d_1\w\nnh\ldots d_m\w$ of positive integers, so
that $\,d_1\w+\ldots+d_m\w=n$, and $\,\jt=0\,$ has the Jor\-dan normal form 
$\,1\ldots1$, each $\,1\,$ being the $\,1\nh\times\hskip-.9pt1\,$ block matrix
$\,[\hh0\hskip.2pt]$, while $\,n\,$ is the sin\-gle\hh-block Jor\-dan normal
form of a {\it generic\/} nil\-po\-tent $\,(1,1)\,$ tensor in dimension $\,n$. 
The Jor\-dan normal form 
$\,2\ldots2\,$ characterizes the case
$\,\mathrm{Ker}\,\jt=\mathrm{Im}\,\jt$. 

If an algebraically constant nil\-po\-tent $\,(1,1)\,$ tensor $\,\jt\,$ on an 
$\,n$-man\-i\-fold has the Jor\-dan normal form 
$\,d_1\w\nnh\ldots d_m\w$, then each sub\-bun\-dle
$\,\mathrm{Ker}\,\jt^i$ clearly has the fibre dimension
$\,\mathrm{min}(i,d_1\w)+\ldots+\mathrm{min}(i,d_m\w)$, and hence
\begin{equation}\label{rkt}
\mathrm{rank}\,\jt^i\hs=\,\,n\,-\,\mathrm{min}(i,d_1\w)\,-\,\ldots\,
-\,\mathrm{min}(i,d_m\w)\hh.
\end{equation}
Let us fix an algebraically constant nil\-po\-tent $\,(1,1)\,$ tensor 
$\,\jt\,$ on an $\,n$-man\-i\-fold $\,M\,$ and a local frame 
field realizing the Jor\-dan 
normal form $\,d_1\w\nnh\ldots d_m\w$ of $\,\jt$. (See 
Remark~\ref{gstrc}.)
We focus on 
three (not necessarily distinct) $\,\jt$-or\-bits
\begin{equation}\label{fbl}
(e_1\w,\dots,e_p\w),\,(\tilde
e_1\w,\dots,\tilde e_q\w), \,(\hat e_1\w,\dots,\hat e_r\w)\hh,
\end{equation}
by which we mean portions of our frame
field corresponding to three of the entries $\,d_1\w,\ldots,d_m\w$. 
Setting $\,e\hn_i\w=\tilde e\nh_j\w=\hat e\hn_k\w=0\,$ for nonpositive
integers $\,i,j,k$, we obtain 
$\,(\jt e\hn_i\w,\jt\tilde e\nh_j\w,\jt\hat e\hn_k\w)
=(e\hn_{i-1}\w,\tilde e\nh_{j-1}\w,\hat e\hn_{k-1}\w)\,$ for all 
integers $\,i,j,k\,$ not exceeding, respectively, $\,p,q\,$ or $\,r$. Finally, 
we denote by $\,C_{i,j}^k$ the coefficient of $\,\hat e\hn_k\w$ in
the expansion of the Lie bracket $\,[\hs e\hn_i\w,\tilde e\nh_j\w]\,$ as a
(functional) combination of our fixed frame, and also set 
$\,C_{i,j}^k=0\,$ if $\,k>r\,$ or one of $\,i,j,k\,$ is nonpositive;
thus $\,C_{i,j}^k$ is well defined for integers $\,i,j,k\,$ with
$\,i\le p\,$ and $\,j\le q$. Now $\,N\nnh=\hh0\,$ in (\ref{nih}) if and only if
\begin{equation}\label{cij}
C_{i,j}^k\hh+\,C_{i-1,j-1}^{k-2}\,
=\,\hs C_{i-1,j}^{k-1}\hh+\,C_{i,j-1}^{k-1}\,\mathrm{\ whenever\ }\,k\ge3\hh,\,\,i\le p\,\mathrm{\
and\ }\,j\le q\hh,
\end{equation}
with $\,C_{i,j}^k=0\,$ for $\,k>r$. Namely,
$\,N(e\hn_i\w,\tilde e\nh_j\w)\,$ evaluated from (\ref{nih}), and then 
projected onto the span of $\,(\hat e_1\w,\dots,\hat e_r\w)$, obviously equals
\begin{equation}\label{cim}
(C_{i-1,j}^k\nnh+C_{i,j-1}^k)\hat e\hn_{k-1}\w\nh
-\hs C_{i-1,j-1}^k\hat e\hn_k\w\hn-\hs C_{i,j}^k\hat e\hn_{k-2}\w\mathrm{\ \
summed\ over\ all\ }\,k\le r.
\end{equation}
The vanishing of the terms involving $\,\hat e_r\w$ (or
$\,\hat e_{r-1}\w$, if $\,r\ge2$) means that $\,C_{i-1,j-1}^{\hs r}\nh=\hn0$
(or, respectively, $\,C_{i-1,j-1}^{\hs r-1}\nh=C_{i-1,j}^{\hs r}\nh
+\hs C_{i,j-1}^{\hs r}$),
both of which are special cases of (\ref{cij}), with $\,k\in\{r+1,r+2\}$.
Leaving these terms aside, we see
that the equality in (\ref{cij}) multiplied by
$\,\hat e\hn_{k-2}\w$ follows if we shift the summation index from $\,k\,$ 
or $\hs k-1\hs$ to $\,k-2\,$
in the first two terms of (\ref{cim}),  
which yields (\ref{cij}) as
$\,\hat e\hn_{k-2}\w=0\,$ unless $\,k\ge3$. 
In terms of $\,E_{i,j}^s=C_{i,j}^{\hs i+j-s+1}\nh$, or
$\,C_{i,j}^k=E_{i,j}^{\hs i+j-k+1}\nh$, (\ref{cij}) can be rewritten as
\begin{equation}\label{eij}
E_{i,j}^s\hh+\hs E_{i-1,j-1}^s\hs=\,E_{i-1,j}^s\hs+\,E_{i,j-1}^s\,\mathrm{\
if\ }\,i+j\ge s+2\hh,\,\,i\le p\,\mathrm{\ and\
}\,j\le q\hh,
\end{equation}
while $\,E_{i,j}^s=0\,$ whenever $\,i+j\ge s+r$. The reason why we prefer to
switch from 
the integer variables $\,i,j,k\,$ to $\,i,j,s\,$ with $\,s=i+j-k+1$, or
$\,k=i+j-s+1$, is that 
(\ref{eij}) uses a fixed value of $\,s$,
allowing us to treat different values of $\,s\,$ as completely unrelated.
Our conclusions may be summarized as follows.
\begin{lemma}\label{noiff}Given\/ $\,\jt\,$ and the frame field as above, 
the Nijen\-huis tensor\/ {\rm(\ref{nih})} vanishes if and only if, for any
ordered triple of not necessarily distinct\/ $\,\jt$-or\-bits\/
{\rm(\ref{fbl})}, one has\/ {\rm(\ref{eij})} along with
\begin{equation}\label{bdr}
\begin{array}{l}
E_{i,j}^s\hs=\hs\,0\,\,\mathrm{\ if\ }\,i+j\hs<s\mathrm{,\ or\
}\,i+j\hs\ge s+r\mathrm{,\
or\ }\,i\le0\mathrm{,\ or\ }\,j\le0\mathrm{,\ and}\\
\mathrm{our\ }\,E_{i,j}^s\mathrm{\ are\ defined\ for\ all\
}\hn\,i,j,s\in\bbZ\nh\,\mathrm{\ such\ 
that\ }\,i\le p\,\mathrm{\ and\ }\,j\le q\hh.
\end{array}
\end{equation}
\end{lemma}
\begin{proof}We already saw that (\ref{eij}) is equivalent to (\ref{cij}),
while (\ref{bdr}) is clearly nothing else than the obvious boundary
conditions ($C_{i,j}^k\hs=0\,$ if $\,k\le0$, or $\,k>r$, or $\,i\le0$, or
$\,j\le0$) coupled with our convention about when $\,C_{i,j}^k$ makes sense.
\end{proof}
Next, $\,\zz^l=\mathrm{Ker}\,\jt^l$ is in\-te\-gra\-ble (which may or may
not be the case) if 
and only if $\,C_{i,j}^k\hs=0\,$ whenever $\,i,j\le l<k$, that is,
\begin{equation}\label{inl}
E_{i,j}^s\hs=\,0\,\mathrm{\ for\ all\ }\,i,j,s\mathrm{ \ with\
}\,i,j\le l\mathrm{ \ and\ }\,i+j\ge s+l\hh,
\end{equation}
and for all ordered triples of (not necessarily distinct) $\,\jt$-or\-bits 
(\ref{fbl}). With the last clause repeated, the in\-te\-gra\-bi\-li\-ty of
$\,\zz^l=\mathrm{Ker}\,\jt^l$ for {\it all\/} $\,l\ge0\,$ clearly amounts to
\begin{equation}\label{ali}
E_{i,j}^s\hs=\,0\,\mathrm{\ whenever\ }\,i,j\ge s\hh,
\end{equation}
since the condition $\,i,j\le i+j-s\,$ is nothing else than $\,i,j\ge s$.
\begin{remark}\label{eijse}The equality in (\ref{eij}) obviously holds, for
all $\,i,j,s\in\bbZ$, if $\,E_{i,j}^s$ is a function of 
$\,i\,$ alone, or of $\,j\,$ alone, or equals $\,i+j\,$ plus a function of
$\,s$. 
\end{remark}

\renewcommand{\theprop}{\thesection.\arabic{prop}}
\setcounter{prop}{0}
\section{Proof of Theorem~\ref{cntrl}: the necessity of
{\rm(\ref{csd})}}\label{nc}
\setcounter{equation}{0}
For the Jor\-dan normal form of an algebraically constant nil\-po\-tent
$\,(1,1)\,$ tensor, {\it not\/} being of type (\ref{csd}) clearly
means that it
\begin{equation}\label{nbg}
\mathrm{contains\ three\ different\ Jor\-dan\ blocks\ of\ lengths\
}\,p,q,r\,\mathrm{\ with\ }\,p\le q<r\hh.
\end{equation}
\begin{prop}\label{ctrol}In any dimension\/ $\,n\ge1$, the condition\/
{\rm(\ref{nbg})}, imposed on the Jor\-dan normal form of an algebraically
constant nil\-po\-tent\/ $\,(1,1)$ tensor\/ $\,\jt$, implies that
the algebraic type of\/ $\,\jt\/$ is not controlled by the Nijen\-huis tensor\/
{\rm(\ref{nih})}. More precisely, $\,\jt\,$ can be realized as a
left-in\-var\-i\-ant\/ $\,(1,1)\,$ tensor on a Lie group, in such a way that\/
$\,N\nnh=\hh0$, but\/ $\,\mathrm{Ker}\,\jt\hh^p$ is non\-in\-te\-gra\-ble for
some integer\/ $\,p\ge1$. One may choose\/ $\,p\,$ to be the shortest block
length in the Jor\-dan normal form of\/ $\,\jt$.
\end{prop}
\begin{proof}We identify a local frame field for $\,\jt$, chosen as in
Sect.\,\ref{lb}, with a basis of a Lie algebra
$\,\mathfrak{g}\,$ formed by left-in\-var\-i\-ant vector fields on a Lie group 
$\,G$. This is achieved by requiring (\ref{cij}) and the boundary conditions 
($C_{i,j}^k\hs=0\,$ if $\,k\le0$, or $\,k>r$, or $\,i\le0$, or $\,j\le0$)
to be satisfied by {\it constants\/} $\,C_{i,j}^k$ or, equivalently, finding
constants $\,E_{i,j}^s$ with (\ref{eij}) -- (\ref{bdr}). 
(Our choice will cause all brackets to lie in the center, thus implying the
Ja\-co\-bi identity.)
Our $\,\jt\,$ then becomes a left-in\-var\-i\-ant $\,(1,1)\,$ tensor 
field on $\,G\,$ acting as an en\-do\-mor\-phism of the tangent bundle which
sends each frame vector field either to the preceding one, or to zero. 
As a consequence of (\ref{nbg}), our local frame contains
\begin{equation}\label{tdf}
\mathrm{three\ different\ }\,\jt\hyp\mathrm{or\-bits\ (\ref{fbl})\ of\
lengths\ }\,p,q,r\,\mathrm{\
with\ }\,p\le q<r\hh.
\end{equation}
Fixing such $\,\jt$-or\-bits, we now set, in 
the discussion of Sect.\,\ref{lb}, $\,E_{i,j}^s\nh=0\,$ for all integers
$\,i,j,s$, with the exception of $\,(i,j,p)\,$ from the set 
$\,[1,p\hh]\times[1,q]\times\{p\}\,$ contained in the range 
$\,[1,p\hh]\times[1,q]\times[1,r]\,$ corresponding to our three 
$\,\jt$-or\-bits (\ref{tdf}).
\begin{figure}[H]
  \centering
  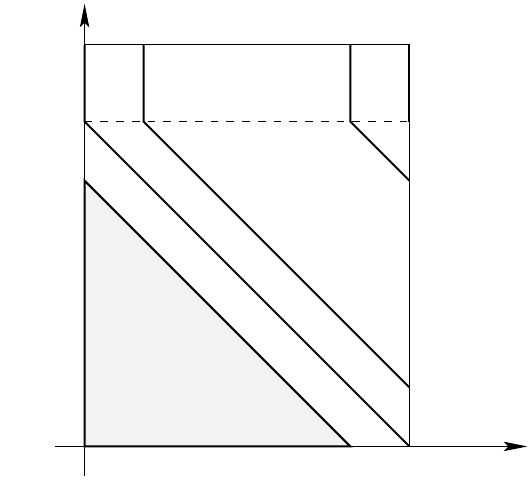
  \caption{Values of \ $E=E_{i,j}^p$}
  \label{fig:ng9}
\end{figure}
Given integers 
$\,i\le p\,$ and $\,j\le q$, we define $\,E_{i,j}^p$ by
\begin{equation}\label{dfe}
E_{i,j}^p\hs
=\,\mathrm{max}\hh(0\hh,\,i\,+\,\hn\mathrm{min}\hh(0\hh,\hs j-p))\hh.
\end{equation}
Speaking below of rectangles, triangles, lines and line segments, we always
mean their intersections with $\,\bbZ^2\nh$, while 
(sub)rec\-tan\-gles are occasionally reduced to segment or single points.
Restricted to $\,(i,j)$ ranging over the rectangle $\,[1,p\hh]\times[1,q]$, 
our $\,E_{i,j}^p$ equals $\,0\,$ on the triangle with vertices
$\,(1,1),(1,p-1),(p-1,1)\,$ (treated as the empty set
when $\,p=1$, or the single point $\,(1,1)\,$ for $\,p=2$), and
$\,E_{i,j}^p=p$ on the segment $\,\{p\}\times[\hh p,q]\,$ (a point
when $\,p=q$); 
the latter claim is obvious, the former immediate from the equality
\begin{equation}\label{ipm}
i\,+\,\hn\mathrm{min}\hh(0\hh,\hs j-p)\,=\,\mathrm{min}\hh(i\hh,\hs i+j-p)\hh.
\end{equation}
If $\,p>1$, then, for any $\,l\in\{1,\dots,p-1\}$, we have $\,E_{i,j}^p=l\,$
on the 
two-seg\-ment broken line joining the points $\,(l,q),(l,p),(p,l)\,$
(reduced to a segment when $\,p=q>1$); cf.\ (\ref{ipm}). 
(This is particularly simple for %
$\,p=q=1$, with $\,E_{1,1}^1\nh=1$.)

The corresponding Nijen\-huis tensor (\ref{nih}) vanishes identically, by 
Lemma~\ref{noiff}, since -- as we now proceed to show -- our $\,E_{i,j}^p$
satisfy (\ref{eij}) and (\ref{bdr}). 
First, (\ref{bdr}) holds, as nonpositivity of $\,i,j\,$ or 
$\,i+j-s=i+j-p\,$ in
(\ref{dfe}) yields $\,E_{i,j}^p\nh=0$, by (\ref{ipm}), and the remaining
implication is 
vacuous: $\,i+j\le p+q=s+q<\nh s+r$.

Next, (\ref{eij}) ``essentially'' follows from Remark~\ref{eijse}:
$\,E_{i,j}^p\nh=\mathrm{max}\hh(0,i)$, which is a function of $\,i$, 
on the sub\-rec\-tan\-gle
$\,[1,p\hh]\times[\hh p,q]\,$ (a segment when $\,p=q$).
On $\,[1,p\hh]\times[1,p\hh]$, (\ref{dfe}) in turn gives
$\,E_{i,j}^p\nh=\mathrm{max}\hh(0,\,i+j-p)$, which coincides with
$\,\,i+j-p\,$ on the sub\-tri\-an\-gle given by $\,i+j\ge s+2=p+2$.

To dispel any doubts, we now establish (\ref{eij}) rigorously, for $\,s=p$.
Of interest to us are integers $\,i,j\,$ with $\,i+j\ge p+2$, $\,\,i\le p\,$
and 
$\,j\le q$. We are also free to assume that $\,i,j\ge1$, since otherwise,
by (\ref{bdr}), all four terms in (\ref{eij}) equal $\,0$. If $\,j>p$,
(\ref{dfe}) gives $\,E_{i,j}^p\nh=E_{i,j-1}^p\nh=i\,$ for all $\,i\ge0$.
Now the four terms in (\ref{eij}) are $\,i,i-1,i-1,i\,$ (whenever 
$\,i\ge1$), and the required equality follows. When $\,j\le p\,$ (and hence
$\,j-1<p$), given $\,i\in\bbZ$, (\ref{dfe}) reads
$\,E_{i,j}^p\nh=\mathrm{max}\hh(0\hh,\,i+j-p)$ and, similarly,
$\,E_{i,j-1}^p\nh=\mathrm{max}\hh(0\hh,\,i+j-p-1)\,$ with $\,j\,$ replaced
by $\,j-1$. In the case of interest to us, $\,i+j\ge p+2\,$ (see the beginning
of this paragraph), so that
$\,E_{i,j}^p\nh=i+j-p\,$ and $\,E_{i,j-1}^p\nh=i+j-p-1$, for all $\,i$, 
and we get the equality in (\ref{eij}):
$\,(i+j-p)+(i-1+j-1-p)=(i-1+j-p)+(i+j-1-p)$.

Finally, since $\,E_{p,p}^p\nh=p\ne0$, (\ref{inl}) applied to
$\,i=j=s=l=p\,$ shows that $\,\mathrm{Ker}\,\jt\hh^p$ is not
in\-te\-gra\-ble.
\end{proof}

\section{Proof of Theorem~\ref{cntrl}: the sufficiency of
{\rm(\ref{csd})}}\label{sc}
\setcounter{equation}{0}
We now show that, given an algebraically constant nil\-po\-tent $\,(1,1)\,$ 
tensor $\,\jt$ on a manifold $\,M\,$ of dimension $\,n\ge1$, with
$\,N\nnh=\hh0$, and with the Jor\-dan normal form $\,d_1\w\nnh\ldots d_m\w$
satisfying condition
(\ref{csd}), $\,\jt\,$ must also have the property (i) in
Theorem~\ref{npjor}, and hence 
be locally constant. To this end, we choose a local frame field realizing the
Jor\-dan normal form of $\,\jt$. See Remark~\ref{gstrc}.

In the first case of (\ref{csd}), $\,d_1\w=\ldots=d_m\w=d\,$ for some
$\,d\ge1$, and our local frame field splits into disjoint $\,\jt$-or\-bits 
of the form $\,v_1\w,\dots,v\nh_d\w$, all of length $\,d$, while 
$\,v_i\w=\jt^{d-i}\nh v\nh_d\w$ for $\,i=1,\dots,d-1$, and the final vector
$\,v\nh_d\w$ lies outside of $\,\mathrm{Im}\,\jt$. Thus, 
$\,\zz^i\nh$, $\,i\ge0$, is obviously equal to either $\,T\nh M\,$ (when
$\,i\ge d$), or to $\,\bz\hh^{d-i}$ (if $\,1\le i<d$), and (\ref{ziz}\hh-a)
yields our claim.

Consider now the second case of (\ref{csd}): 
$\,d_1\w=\ldots=d_{m-1}\w=d>d'\nh=d_m\w$ for some $\,d,d'\nh\ge1$, with
$\,m>1$, 
leading to $\,\jt$-or\-bits $\,v_1\w,\dots,v\nh_d\w$ of length $\,d$,
of which there are $\,m-1$, and to one $\,\jt$-or\-bit of length $\,d'\nh$.

We first prove the in\-te\-gra\-bi\-li\-ty of
$\,\zz^i$ when $\,1\le i\le d'\nh$, using induction on $\,i$. As $\,\zz\nh^1$
is spanned by the $\,m\,$ initial vectors from all $\,\jt$-or\-bits, taken one
from each, 
and $\,\bz\hh^{d-1}$ by the $\,m-1$ initial vectors from all $\,\jt$-or\-bits
of length $\,d$, the latter sub\-bun\-dle of $\,T\nh M\,$ is
contained in the former with co\-dimen\-sion
one. Thus, (\ref{ziz}\hh-b) and Lemma~\ref{cdone} 
yield the in\-te\-gra\-bi\-li\-ty of $\,\zz\nh^1\nh$. For the induction step,
if $\,1\le i<d'$ and $\,\zz^i$ is in\-te\-gra\-ble, $\,\zz^{i+1}$ 
is spanned by $\,m(i+1)\,$ vectors: $\,v_1\w,\dots,v_{i+1}\w$ from all the
$\,\jt$-or\-bits combined (if one writes the $\,\jt$-or\-bits as
$\,v_1\w,\dots,v\nh_d\w$ or $\,v_1\w,\dots,v\nh_{d'}\w$), and so $\,\zz^{i+1}$
contains, with co\-dimen\-sion one, the span
$\,\zz^i\nh+\hs\bz\hh^{d-i-1}$ of $\,\zz^i$ and $\,\bz\hh^{d-i-1}\nh$.
By (\ref{ziz}\hh-b) and (\ref{ziz}\hh-d)
$\,[\zz^i\nh+\hs\bz\hh^{d-i-1}\nh,\zz^{i+1}]
\subseteq\zz^{i+1}\hskip-1.1pt+\hs\bz\hh^{d-i-1}\subseteq\zz^{i+1}\nh$, and
Lemma~\ref{cdone} completes the induction step.

Finally, let $\,d'\nh<i<d\,$ and $\,k=d'\nh-1$. This time $\,\zz^i$ contains
$\,\zz^k\nnh\nh+\hs\bz\hh^{d-i}$ with co\-dimen\-sion one: the former is
spanned by $\,(m-1)i+d'$ vectors (the initial $\,i\,$ ones from all 
$\,\jt$-or\-bits of length $\,d$, plus the whole $\,\jt$-or\-bit of length
$\,d'$), the latter -- by the same vectors except the last one in the
length $\,d'$ orbit. Once again, (\ref{ziz}\hh-b) and (\ref{ziz}\hh-d) give 
$\,[\zz^k\nh+\hs\bz\hh^{d-i}\nh,\zz^i]
\subseteq\zz^i\nh+\hs\bz\hh^{d-i}\subseteq\zz^i\nh$, and we can use
Lemma~\ref{cdone}.

\section{Generalized al\-most-tan\-gent structures}\label{ga}
\setcounter{equation}{0}
The following construction provides -- as shown below -- a local description 
of all algebraically constant $\,(1,1)\,$ tensors $\,\jt\,$ such that
$\,\jt^2\nh=0\,$ and the Nijen\-huis tensor (\ref{nih}) vanishes
identically.

Given a distribution $\,\dz\,$ on a manifold
$\,\vs$, let $\,M\,$ be the total space of an af\-fine 
bundle over $\,\vs\,$ associated with the quotient vector bundle 
$\,T\nh\vs/\dz$. Using the bundle projection $\,\pi:M\to\vs\,$ and the
quo\-tient-bun\-dle projection morphism
$\,T\nh\vs\ni v\mapsto[v]\in T\nh\vs/\dz$, we define a 
$\,(1,1)\,$ tensor $\,\jt\,$ on $\,M\,$ by
\begin{equation}\label{txv}
\jt\nh_x\w v
=[d\pi\nh_x\w v]\in T\hskip-2.7pt_y\w\hs\vs/\dz\nnh_y\w
=T\hskip-2.7pt_x\w\hn M\nnh_y\w\mathrm{,\ if\ }\,x\in M\nnh_y\w
=\pi^{-\nnh1}(y)\hh,
\end{equation}
whenever $\,x\in M\,$ and $\,v\in T\hskip-2.7pt_x\w\hn M\nh$.
Then $\,\jt^2\nh=0$, since all $\,\jt$-im\-ages are vertical. Also,
$\,N\nnh=\hh0\,$ in (\ref{nih}). In fact, $\,\mathrm{Im}\,\jt\,$ is the
vertical distribution $\,\vz=\,\mathrm{Ker}\hskip2.3ptd\pi$. Evaluating
(\ref{nih}), withous loss of generality, on $\,\pi$-pro\-ject\-a\-ble vector
fields, we see that, by (\ref{prj}), the first, second and fourth terms on
the right-hand side of vanish as $\,\jt^2\nh=0$. So does the third term:
$\,\jt v,\,\jt w\,$ restricted to each fibre are af\-fine-space
translations, and consequently commute.
\begin{theorem}\label{sqezr}Every algebraically constant\/
$\,(1,1)\,$ tensor\/ $\,\jt\,$ with\/ $\,\jt^2\nh=0\,$ and vanishing
Nijen\-huis tensor\/ {\rm(\ref{nih})} arises, locally, from the above
construction, and the fibre dimension of\/
$\,\dz\,$ equals the co\-dimen\-sion of\/ $\,\mathrm{Im}\,\jt\,$ in\/
$\,\mathrm{Ker}\,\jt$, while
\begin{equation}\label{tin}
\jt\,\mathrm{\ is\ in\-te\-gra\-ble\ if\ and\ only\ if\ so\ is\ the\ 
distribution\ }\,\dz.
\end{equation}
\end{theorem}
\begin{proof}Suppose that $\,\jt^2\nh=0\,$ and $\,N\nnh=\hh0\,$ in
(\ref{nih}). By (\ref{ziz}\hh-a), $\,\mathrm{Im}\,\jt\,$ is an 
in\-te\-gra\-ble distribution, while
$\,\mathrm{Im}\,\jt\subseteq\mathrm{Ker}\,\jt$. 
Due to (\ref{prj}) and (\ref{nih}) with
$\,\jt^2\nh=0$,
\begin{equation}\label{cmi}
\mathrm{any\ two\ }\,(\mathrm{Im}\,\jt)\hyp\mathrm{pro\-ject\-a\-ble\ vector\
fields\ have\ commuting\ }\,\jt\hyp\mathrm{im\-ages.}
\end{equation}
By (\ref{ziz}\hh-b) for $\,i=j=1\,$ and
Lemma~\ref{prdis}(c), 
on an open set $\,M'\nh\subseteq M\,$ with a bundle projection
$\,\pi:M'\nh\to\vs\,$ having $\,\mathrm{Im}\,\jt\,$ as the vertical
distribution, $\,\mathrm{Ker}\,\jt$ is $\,\pi$-pro\-ject\-a\-ble onto a
distribution $\,\dz\,$ on $\,\vs$, with (\ref{tin}) obvious from
Theorem~\ref{npjor} and (\ref{pri}). Any $\,\pi$-pro\-ject\-a\-ble lift, along
the fibre $\,\pi^{-\nnh1}(y)$, of any vector $\,w\,$ tangent to $\,\vs\,$ at 
$\,y\in\vs$, is mapped by $\,\jt\,$ onto the ``vertical lift'' of $\,w$, a
vector field tangent to $\,\pi^{-\nnh1}(y)$, which vanishes precisely when
$\,w\,$ is tangent to $\,\dz$. By (\ref{cmi}) 
the vertical lifts of any $\,w,w'\in T\hskip-2.7pt_y\w\hs\vs\,$ commute. 
This turns 
$\,\pi^{-\nnh1}(y)$, locally, into an af\-fine space having the translation
vector space $\,T\hskip-2.7pt_y\w\hs\vs/\dz\nnh_y\w$, with $\,\jt\,$ given
by (\ref{txv}).
\end{proof}
Theorem~\ref{sqezr} illustrates a special case of Theorem~\ref{cntrl}:
the condition $\,\jt^2\nh=0$ corresponds to the Jor\-dan
normal forms $\,2\ldots2\,$ and $\,2\ldots21\ldots1\,$ (plus 
$\,1\ldots1$, for $\,\jt=0$). Of these,
only $\,2\ldots2$, $\,\,2\ldots21\,$ and $\,1\ldots1\,$ satisfy (\ref{csd}),
reflecting the fact that $\,\dz\,$ is necessarily in\-te\-gra\-ble only if 
it has the fibre dimension $\,0,\hs1\,$ or $\,\dim\vs$.

When $\,\mathrm{Ker}\,\jt=\mathrm{Im}\,\jt$, that is, $\,\dz\,$ is the zero
distribution, our construction gives rise to what is referred to as {\it
al\-most-tan\-gent structures\/} \cite{yano-davies,goel}, and
Theorem~\ref{sqezr} becomes the local version of 
\cite[Theorem on p.\,69]{crampin-thompson}.

\section{Differential $\,q$-forms on an $\,n$-man\-i\-fold,
$\,q=0,1,2,n-1,n$}\label{df}
\setcounter{equation}{0}
We now prove Proposition~\ref{dffrm}. Let $\,\zeta\hh$ be an
algebraically constant differential $\,q$-form on an 
$\,n$-di\-men\-sion\-al manifold, $\,q=0,1,2,n-1,n$, with $\,d\hh\zeta=0\,$
(the last condition being obviously redundant if $\,q=n\,$ or -- as
$\,\zeta\,$ is constant -- if $\,q=0$).

The cases $\,q=0\,$ and $\,q=1\,$ are obvious: the $\,1$-form $\,\zeta\,$ (if
nonzero), being 
locally exact, equals $\,dx^1$ in suitable local coordinates
$\,x^1\nh,\dots,x^n\nh$.

When $\,q=2$, algebraic constancy amounts to constant rank, and our claim
follows as Dar\-boux's theorem
\cite[p.\,40]{bryant-chern-gardner-goldschmidt-griffiths} gives
$\,\zeta=dx^1\nnh\wedge dx^2\nh+\ldots+\hs dx^{2r-1}\nnh\wedge dx^{2r}$ in
some local coordinates $\,x^1\nh,\dots,x^n\nh$, with
$\,2r=\mathrm{rank}\,\zeta\ge0$.

If $\,q=n\,$ and $\,\zeta\ne0$, we have, in suitable 
local coordinates $\,x^1\nh,\dots,x^n\nh$,
\begin{equation}\label{zed}
\zeta=dx^1\nnh\wedge dx^2\nnh\wedge\ldots\wedge\hs dx^n\nh\mathrm{,\ where\
}\,x^2\nh,\dots,x^n\mathrm{\ can\ be\ arbitrary,}
\end{equation}
as long as $\,dx^2\nh\wedge\ldots\wedge dx^n\ne0$. In
fact, starting from
$\,\zeta=\phi\,dx^1\nnh\wedge\hs dx^2\nh\wedge\ldots\wedge\hs dx^n$ for a
function $\,\phi\,$ without zeros, and choosing $\,\psi\,$ with
$\,\partial\nh_1\w\psi=\phi$, we see that $\,d\psi\,$ equals
$\,\phi\,dx^1$ plus a functional combination of
$\,dx^2\nh,\dots,dx^n$ and so
$\,\zeta=d\psi\wedge dx^2\nh\wedge\ldots\wedge dx^n\nh$.

Finally, let $\,q=n-1$. Assuming $\,\zeta\hh$ to be nonzero, and 
fixing a nonzero $\,n$-form $\,\omega$, we get
$\,\zeta=\omega(v,\,\cdot\,,\dots,\,\cdot\,)$, for a unique (nonzero) vector
field $\,v$. Then, by (\ref{zed}), 
$\,\omega=dx^1\nnh\wedge\hs dx^2\nh\wedge\ldots\wedge\hs dx^n$ in some 
local coordinates $\,x^1\nh,\dots,x^n\nh$, with $\,x^2\nh,\dots,x^n$
chosen so that $\,dx^2\nh(v)=\ldots=dx^n(v)=0$. 
Now $\,\zeta=\chi\,dx^2\nnh\wedge\ldots\wedge\hs dx^n$ for $\,\chi=dx^1\nh(v)$,
and $\,\partial\nh_1\w\chi=0\,$ as $\,d\hh\zeta=0$. Our $\,\zeta$, being
thus a top-de\-gree form in $\,n-1\,$ variables, equals, 
by (\ref{zed}), $\,dy^2\nh\wedge\ldots\wedge dy^n$ in suitable 
coordinates $\,y^1\nh,\dots,y^n\nh$.

\section{Differential forms of other degrees}\label{do}
\setcounter{equation}{0}
We now proceed to verify the statement preceding formula (\ref{nin}).
The algebraic constancy of $\,\zeta\hh$ is clear as
$\,\zeta=
(\xi\hh^1\nnh\wedge\xi\hh^2\nh+\xi\hh^3\nnh\wedge\xi\hh^4)\wedge\xi\hh^5\wedge\ldots
\wedge\xi\hh^{q+2}\nh$, with linearly independent $\,1$-forms 
$\,\xi\hh^1\nh,\dots,\xi\hh^{q+2}\nh$, and its closedness since
$\,d\hh\zeta\,$ is the exterior product of 
$\,(dx^1\nnh\wedge dx^2\nh+dx^3\nnh\wedge dx^4)
\wedge(dx^1\nnh\wedge dx^2\nh-dx^3\nnh\wedge dx^4)\,$ (obviously equal to
$\,0$) and $\,dx^6\nnh\wedge\ldots\wedge dx\hh^{q+2}\nh$. Being algebraically
constant, $\,\zeta\hh$ gives rise to the vector sub\-bun\-dle
$\,\mathcal{F}\hs$ of $\,T^*\hskip-1.8ptM\,$ such that the sections of
$\,\mathcal{F}\hs$ are 
those $\,1$-forms $\,\xi\,$ for which
$\,\xi\wedge\zeta=0$. The sections $\,\xi\,$ of $\,\mathcal{F}\hs$ also
coincide with functional combinations of the $\,1$-forms
\begin{equation}\label{fco}
\eta,\,dx^6\nh,\ldots,\,dx\hh^{q+2}\nh\mathrm{,\ where\
}\,\eta=dx^5\nnh+x^1dx^2\nh-x^3dx^4\nh.
\end{equation}
In fact, writing $\,\xi=\xi_{\hh i}\w\hs dx^i\nh$, we see that 
$\,\xi\wedge\zeta\,$ contains no contributions from the terms
$\,\xi_{\hh i}\w\hs dx^i$ (no summation) with $\,6\le i\le q+2\,$ (making
$\,\xi\hh_6\w,\dots,\xi\hh_{q+2}\w$ completely arbitrary) while for 
$\,\theta
=(dx^1\nnh\wedge dx^2\nh+dx^3\nnh\wedge dx^4)\wedge(dx^5\nnh+x^1dx^2\nh
-x^3dx^4)\,$ one has
\[
\theta
=dx^1\nnh\wedge dx^2\nnh\wedge dx^5\nh+dx^3\nnh\wedge dx^4\nnh\wedge dx^5\nh
-x^3dx^1\nnh\wedge dx^2\nnh\wedge dx^4\nh
+x^1dx^2\nnh\wedge dx^3\nnh\wedge dx^4\nh.
\]
and so each term $\,\xi_{\hh i}\w\hs dx^i$ (no summation, again) with
$\,q+2<i\le n\,$
contributes to $\,\xi\wedge\zeta\,$ the expression
$\,\xi_{\hh i}\w\hs dx^i\nh\nnh\wedge\theta\nh\wedge
dx^6\nnh\wedge\ldots\wedge dx\hh^{q+2}$ (no summation) comprising all the
terms in $\,\xi\wedge\zeta\,$ involving the factor $\,dx^i\nh$. Linear 
independence of the differentials $\,dx^1\nh,\dots,dx^n$ now gives
$\,\xi_{\hh i}\w=0\,$ whenever $\,q+2<i\le n$. Finally, the exterior
products of
$\,\xi_1\w dx^1\nh,\,\xi_{\hh2}\w dx^2\nh,\,\xi_{\hh3}\w dx^3\nh,
\,\xi_{\hh4}\w dx^4\nh,\,\xi_{\hh5}\w dx^5$ with $\,\theta\,$ are 
\begin{equation}\label{xtp}
\begin{array}{l}
\xi_1\w(dx^1\nnh\wedge dx^3\nnh\wedge dx^4\nnh\wedge dx^5\nh
+x^1dx^1\nnh\wedge dx^2\nnh\wedge dx^3\nnh\wedge dx^4)\hh,\\
\xi_{\hh2}\w(dx^2\nnh\wedge dx^3\nnh\wedge dx^4\nnh\wedge dx^5)\hh,\\
\xi_{\hh3}\w(dx^1\nnh\wedge dx^2\nnh\wedge dx^3\nnh\wedge dx^5\nh
-x^3dx^1\nnh\wedge dx^2\nnh\wedge dx^3\nnh\wedge dx^4)\hh,\\
\xi_{\hh4}\w(dx^1\nnh\wedge dx^2\nnh\wedge dx^4\nnh\wedge dx^5)\hh,\\
\xi_{\hh5}\w(x^3dx^1\nnh\wedge dx^2\nnh\wedge dx^4\nnh\wedge dx^5\nh
-x^1dx^2\nnh\wedge dx^3\nnh\wedge dx^4\nnh\wedge dx^5)\hh.
\end{array}
\end{equation}
The condition $\,\xi\wedge\zeta=0\,$ means, after the cancellation of 
$\,dx^6\nnh\wedge\ldots\wedge dx\hh^{q+2}\nh$, that the sum of the five lines
of (\ref{xtp}) equals $\,0$. Writing $\,[ijkl]\,$ for
$\,dx^i\nnh\wedge dx^j\nnh\wedge dx^k\nnh\wedge dx^l\nh$, we see that
$\,[1345]\,$ and $\,[1235]\,$ occur just once each, giving
$\,\xi_1\w=\xi_{\hh3}\w=0$, while the sum of the remaining three lines equals
$\,(\xi_{\hh4}\w+\xi_{\hh5}\w x^3)[1245]
+(\xi_{\hh2}\w-\xi_{\hh5}\w x^1)[2345]$. Thus,
$\,\xi_{\hh4}\w+\xi_{\hh5}\w x^3=0=\xi_{\hh2}\w-\xi_{\hh5}\w x^1\nh$, and 
the sum of
$\,\xi_{\hh i}\w\hs dx^i$ over $\,i=1,\dots,5\,$ equals a function times
the $\,1$-form $\,\eta\,$ in (\ref{fco}), proving our claim about (\ref{fco}).

If $\,\zeta\hh$ were in\-te\-gra\-ble, so would be -- according to 
(\ref{iii}) -- the simultanous kernel of the
$\,1$-forms (\ref{fco}) (that is, of all sections of $\,\mathcal{F}$),
naturally determined by $\,\zeta$. 
This is not the case, as
$\,d\eta\wedge\eta\wedge dx^6\nnh\wedge\ldots\wedge dx\hh^{q+2}$ is nonzero, 
being equal to 
$\,(dx^1\nnh\wedge dx^2\nnh\wedge dx^5\nh
-dx^3\nnh\wedge dx^4\nnh\wedge dx^5\nh
-x^3dx^1\nnh\wedge dx^2\nnh\wedge dx^4\nh
-x^1dx^2\nnh\wedge dx^3\nnh\wedge dx^4)
\wedge dx^6\nnh\wedge\ldots\wedge dx\hh^{q+2}\nh$.

\section{Sym\-me\-tric $\,(0,2)\,$ and $\,(2,0)\,$ tensors}\label{ss}
\setcounter{equation}{0}
\begin{proof}[Necessity and sufficiency of (\ref{nas})]Let $\,g\,$ be
in\-te\-gra\-ble, with $\,\nabla\nh g=0$ for a fixed 
tor\-sion-free connection $\,\nabla\nh$. The in\-te\-gra\-bi\-li\-ty of the
distribution $\,\vz\nh=\mathrm{Ker}\,g$, due to (\ref{iii}), allows us to
choose local coordinates and index ranges for $\,i,a,\lambda,\mu,\nu\,$ as in
Remark~\ref{nwcon}, so that $\,\vz\,$ is spanned by the coordinate vector
fields $\,\partial\nh_a\w$. As $\,\vz$ is obviously
$\,\nabla\nh$-par\-al\-lel, $\,\vg_{\hskip-2.7ptia}^{\hs k}=
\vg_{\hskip-2.7ptab}^{\hs k}=0$, while $\,g_{ia}\w=g_{ab}\w=0$, 
so that $\,\partial\nh_a\w g_{ij}\w=\vg_{\hskip-2.7ptai}^{\hs k}g_{kj}\w
+\vg_{\hskip-2.7ptaj}^{\hs k}g_{ik}\w=0$, and pro\-ject\-a\-bil\-i\-ty
of $\,g\,$ along $\,\vz\,$ follows from (\ref{prt}).

Conversely, suppose that $\,g\,$ is pro\-ject\-a\-ble along the
in\-te\-gra\-ble distribution $\,\vz\nh=\mathrm{Ker}\,g$. As before, we invoke
Remark~\ref{nwcon}, selecting local coordinates with index ranges for
$\,i,a,\lambda,\mu,\nu\,$ so as to make $\,\vz\,$ the span of the coordinate
fields $\,\partial\nh_a\w$. As pro\-ject\-a\-bil\-i\-ty of $\,g\,$ along
$\,\vz\,$ gives $\,\partial\nh_a\w g_{ij}\w=0$, while
$\,g_{ia}\w=g_{ab}\w=0$, the components $\,g_{ij}\w$ represent a
pseu\-\hbox{do\hs-}Riem\-ann\-i\-an metric in the factor manifold with the
coordinates $\,x^i\nh$. Denoting by $\,\vg_{\hskip-2.7ptij}^{\hs k}$ the
components of its Le\-vi-Ci\-vi\-ta connection, we now use Remark~\ref{nwcon}
to define the required tor\-sion-free 
connection $\,\nabla\hs$ with $\,\nabla\nh g=0$.
\end{proof}
\begin{proof}[Proof of Proposition~\ref{inttz}]In\-te\-gra\-bi\-li\-ty of the
former implies that of the latter by (\ref{iii}). Conversely, let 
the distribution $\,\bz=\mathrm{Im}\,\jt\,$ be in\-te\-gra\-ble. Using
Remark~\ref{nwcon}, we fix local coordinates and index ranges for 
$\,i,a,\lambda,\mu,\nu\,$ so that $\,\bz\,$ is the span of the coordinate 
fields $\,\partial\nh_i\w$. Thus, $\,\jt^{ia}\nh=\jt^{ab}\nh=0$, as the
$\,1$-forms $\,dx^a$ annihilate each $\,\partial\nh_i\w$, and hence are
sections of the sub\-bun\-dle
$\,\mathrm{Ker}\,\jt\subseteq T^*\hskip-1.8ptM\nh$. On each leaf of 
$\,\bz$, the restriction of $\,\jt\,$ is nondegenerate -- see (\ref{ndg}) --
and so it is the reciprocal of a pseu\-\hbox{do\hs-}Riem\-ann\-i\-an metric on 
the leaf. Its Le\-vi-Ci\-vi\-ta connection, with the components 
$\,\vg_{\hskip-2.7ptij}^{\hs k}$ (possibly depending on the variables 
$\,x^a$), makes the restriction of $\,\jt\,$ 
parallel. Thus, we may again invoke Remark~\ref{nwcon}
to obtain a tor\-sion-free 
connection $\,\nabla\hs$ such that $\,\nabla\nh\jt=0$.
\end{proof}

\section{Local constancy of bi\-vec\-tor fields}\label{st}
\setcounter{equation}{0}
\begin{proof}[Proof of Proposition~\ref{crbiv}]The `only if' part is
immediate: for a tor\-sion-free connection $\,\nabla\hs$ on the given manifold
having $\,\nabla\nh\jt=0$, the distribution $\,\bz=\mathrm{Im}\,\jt\,$ is
$\,\nabla\nh$-par\-al\-lel and hence in\-te\-gra\-ble, cf.\ (\ref{iii}), and
the tor\-sion-free connections induced by $\,\nabla\hs$ on the leaves of
$\,\bz\,$ make the restriction of $\,\jt\,$ to each leaf parallel, which
implies the same (and hence also closed\-ness) for their inverses.

Let us now assume that $\,\bz=\mathrm{Im}\,\hs\jt\hs$ is in\-te\-gra\-ble and
the inverses of the restrictions of $\,\jt\,$ to the leaves of $\,\bz\,$ are
all closed. These inverses are symplectic forms $\,\zeta$ on the leaves, and
the Dar\-boux theorem with parameters 
\cite[Lemma 3.10]{bandyopadhyay-dacorogna-matveev-troyanov} allows us to
choose functions $\,x^i$ which, restricted to each leaf, form local
coordinates with
$\,\zeta=dx^1\nnh\wedge dx^2\nh+\ldots+\hs dx^{2r-1}\nnh\wedge dx^{2r}\nh$,
where
$\,2r=\mathrm{rank}\,\zeta=\mathrm{rank}\,\jt$. We may also choose
functions $\,x^a\nh$, with the index ranges $\,1\le i\le 2r<a\le\dim M\nh$,
such that the differentials $\,dx^a$ form a local trivialization 
of the sub\-bun\-dle $\,\mathrm{Ker}\,\jt\subseteq T^*\hskip-1.8ptM\nh$.
In the resulting product coordinates $\,x^i\nh,x^a$ the components of
$\,\jt\,$ are all constant: $\,\jt^{ia}\nh=\jt^{ab}\nh=0$, while 
$\,\jt^{ij}\nh=1\,$ (or, $\,\jt^{ij}\nh=-\nnh1$) if $\,(j,i)$, or 
$\,(i,j)$, is one of the pairs $\,(1,2),(3,4),\dots,(2r-1,2r)$, and 
$\,\jt^{ij}\nh=0\,$ otherwise.
\end{proof}
  
\section{In\-te\-gra\-bi\-li\-ty of the kernels and images}\label{ik}
\setcounter{equation}{0}
For any vector bundle $\,\lz\,$ over a manifold $\,M\,$ and a 
vec\-tor-bun\-dle mor\-phism $\,\jt:T\nh M\to\lz^*$ of constant rank $\,r\,$
into its dual $\,\lz^*\nh$, the resulting dual mor\-phism
$\,\jt\hn^*\nnh:\lz\to T^*\hskip-1.8ptM\nh$, which also has
$\,\mathrm{rank}\,\jt\hn^*\nnh\nh=r$, gives rise to a ten\-sor-like object
$\,\tni$ (specifically, a section of
$\,\mathrm{Hom}\,(\lz\otimes\lz^{\wedge r}\nnh,
\,[T^*\hskip-1.8ptM]^{\wedge(r+2)})$), sending sections
$\,v,v\hn_1\w,\dots,v\hn_r\w$ of $\,\lz\,$ to the $\,(r+2)$-form
\begin{equation}\label{nht}
\tni(v,v\hn_1\w,\dots,v\hn_r\w)
=[d\hs(\jt\hn^*\nnh v)]\wedge\hs\jt\hn^*\nnh v\hn_1\w\wedge\ldots
\wedge\hs\jt\hn^*\nnh v\hn_r\w\hh.
\end{equation}
Here 
$\,d\hs[\jt\hn^*\nnh(f\nh v)]=fd\hs(\jt\hn^*\nnh v)+d\hskip-.8ptf\nh
\wedge\hn\jt\hn^*\nnh v\,$
for a function $\,f\nh$. However, $\,\tni\hs$ itself is ten\-so\-ri\-al:
the non\-ten\-so\-ri\-al term $\,d\hskip-.8ptf\nh\wedge\hn\jt\hn^*\nnh v\,$ in
the last equality has zero exterior product with
$\,\jt\hn^*\nnh v\hn_1\w\wedge\ldots\wedge\hs\jt\hn^*\nnh v\hn_r\w$, since
$\,\mathrm{rank}\,\jt\hn^*\nnh\nh=r$. Furthermore,
\begin{equation}\label{ine}
\tni\nnh=\hh0\,\,\mathrm{\ identically\ if\ and\ only\ if\ 
}\,\mathrm{Ker}\,\jt\,\mathrm{\ is\ in\-te\-gra\-ble.}
\end{equation}
In fact, as $\,\mathrm{Ker}\,\jt\,$ is the simultanous kernel of the 
$\,1$-forms $\,\jt\hn^*\nnh v$, for all sections $\,v$ of $\,\lz$, its
in\-te\-gra\-bi\-li\-ty amounts to $\,d$-clos\-ed\-ness of the ideal
generated by all such $\,1$-forms which, as
$\,\mathrm{rank}\,\jt\hn^*\nnh\nh=r$, 
is nothing else than the vanishing of $\,\tni\nh$.

In the case of $\,\lz=T\nh M\,$ and a (possibly non\-sym\-me\-tric) 
$\,(0,2)\,$ tensor $\,g\,$ of constant rank $\,r\,$ on $\,M\nh$, treated as a
mor\-phism $\,\jt:T\nh M\to T^*\hskip-1.8ptM\,$ sending a vector field 
$\,w\,$ to the $\,1$-form $\,g(\,\cdot\,,w)$, the dual $\,\jt\hn^*\nnh$ acts
via $\,v\mapsto g(v,\,\cdot\,)$. Then
\begin{equation}\label{bcm}
\begin{array}{l}
\tni\hs\mathrm{\ in\ (\ref{nht})\ becomes\ }\,N'\mathrm{\ in\ 
(\ref{nnh}{\hyp}a),\ so\ that\ }\,N'\mathrm{\ is\ ten\-so\-ri\-al.}
\end{array}
\end{equation}
Let $\,\jt\,$ now be an algebraically constant nil\-po\-tent $\,(1,1)\,$
tensor on an $\,n$-man\-i\-fold with the Jor\-dan normal form
$\,d_1\w\nnh\ldots d_m\w$, cf.\ (\ref{wdi}).
In\-te\-gra\-bi\-li\-ty of $\,\jt$, as well as its local constancy, is
equivalent, by Theorem~\ref{npjor}, to the simultaneous vanishing of
the Nijen\-huis tensor $\,N\hs$ in (\ref{nih}) along with further
$\,d_1\w\nh-1\,$ Nijen\-huis-type tensors $\,N^i$, where $\,1\le i<d_1$, such
that $\,N^i\nh=0\,$ if and only if $\,\zz^i\nh=\mathrm{Ker}\,\jt^i$ is
in\-te\-gra\-ble. Specifically, this follows from (\ref{ine}) if we define
$\,N^i$ to be $\,\tni\hs$ in
(\ref{nht}) with $\,\lz=T\nh M\,$ and $\,\jt\,$ replaced by $\,\jt^i\nh$,
where $\,r\,$ equals (\ref{rkt}), and a fixed Riemannian metric on $\,M\,$ has
been used to identify $\,T\nh M\,$ with $\,T^*\hskip-1.8ptM\nh$, thus turning
each $\,\jt^i$ separately into a vec\-tor-bun\-dle mor\-phism
$\,T\nh M\to T^*\hskip-1.8ptM=\lz^*\nh$.

Finally, given a (skew)sym\-me\-tric $\,(2,0)\,$ tensor $\,\jt\,$ of 
constant rank $\,r\,$ on a manifold $\,M\nh$, we associate with $\,\jt\,$ a
Nijen\-huis-type $\,(2r+3,0)\,$ tensor $\,\hni\nh$, testing the 
in\-te\-gra\-bi\-li\-ty of the image distribution 
$\,\vz=\mathrm{Im}\,\hs\jt\subseteq T\nh M\nh$. (Note that $\,\jt\,$ is a
bundle mor\-phisms $\,T^*\hskip-1.8ptM\to T\nh M\,$ acting via
$\,\xi\mapsto\jt\hh\xi=\jt(\,\cdot\,,\xi)$.) To define $\,\hni\nh$,
we again fix a Riemannian metric on $\,M\nh$, which allows us to use
contractions and the Hodge star operator $\,*\,$ (as the latter enters our 
formula quadratically, $\,M\,$ need not be oriented). With 
$\,\jt\hh\xi=\jt(\,\cdot\,,\xi)\,$ as above, for $\,1$-forms
$\,\xi\,$ on $\,M\nh$, we set
\begin{equation}\label{hni}
\hni(\xi,\xi^1\nh,\dots,\xi^r\nh,\eta,\eta^1\nh,\dots,\eta^r)\,
=\,\Omega[\jt\hh\xi,\jt\eta]\hh,
\end{equation}
where $\,\xi,\xi^1\nh,\dots,\xi^r\nh,\eta,\eta^1\nh,\dots,\eta^r$ are any
$\,1$-forms on $\,M\nh$, and
\begin{equation}\label{omg}
\begin{array}{l}
\Omega\,\mathrm{\hs\hs\ denotes\hs\hs\ the\hs\hs\ result\hs\hs\ of\hs\hs\
an\hs\hs\ 
}\,(r\hs-1)\hyp\mathrm{fold\hs\hs\ contraction}\\
\mathrm{(\ref{rmo})\ of\ 
}\hskip-2.7pt*\hskip-2pt(\jt\hh\xi^1\nnh\wedge\ldots\wedge\jt\hh\xi^r)\,\,
\hn\mathrm{against\hn\
}\hskip-2.8pt*\hskip-2pt(\jt\eta^1\nnh\wedge\ldots\wedge\jt\eta^r)\nh.
\end{array}
\end{equation}
Clearly, at points where the $\,r$-tuples
$\,\jt\hh\xi^1\nh,\dots,\jt\hh\xi^r$ and 
$\,\jt\eta^1\nh,\dots,\jt\eta^r$ of vector fields are both linearly
independent, $\,\Omega\,$ is, by Remark~\ref{hodge}, a nonzero functional
multiple of the or\-thog\-onal projection onto the or\-thog\-onal complement
of $\,\vz\,$ and, applied
to the Lie brackets $\,[\jt\hh\xi,\jt\eta]$, tests the 
in\-te\-gra\-bi\-li\-ty of $\,\vz\nh$.

\section{Twice-co\-var\-i\-ant symmetric tensors}\label{tc}
\setcounter{equation}{0}
The ten\-so\-ri\-al\-i\-ty of $\,N'$ in (\ref{nnh}\hs-a) was established in
(\ref{bcm}). For 
$\,N''\nh$, since
\begin{equation}\label{lfv}
[\hn\Lie\hskip-1.2pt_{\phi v}\w g](w,u)
=\phi\hh[\Lie\hskip-1.2pt_v\w g](w,u)+(d_w\w\phi)g(v,u)+(d_u\w\phi)g(w,v)
\end{equation}
for any function $\,\phi\,$ on the given manifold $\,M\nh$, the resulting 
non\-ten\-so\-ri\-al contribution to (\ref{nnh}\hs-b) equals the sum 
$\,(d_w\w\phi)g(u,\,\cdot\,)+(d_u\w\phi)g(w,\,\cdot\,)\,$ of the last two
terms in (\ref{lfv}). Its exterior product
with $\,g(v_1\w,\,\cdot\,)\wedge\ldots\wedge\hs g(v_r\w,\,\cdot\,)$ vanishes, 
being a sum of $\,(r+1)$-fold exterior products of sections of a rank
$\,r\hs$ sub\-bundle of $\,T^*\hskip-1.8ptM\nh$, namely, the image of the
mor\-phism sending each vector field $\,v\,$ to $\,g(v,\,\cdot\,)$.
\begin{proof}[Proof of Theorem~\ref{twoni}]We derive our conclusion from
(\ref{nas}), by showing that the vanishing of $\,N'$ (or 
$\,N''$), is equivalent to the in\-te\-gra\-bi\-li\-ty of the distribution 
$\,\vz\nh=\mathrm{Ker}\,g\,$ (or, respectively, to
pro\-ject\-a\-bil\-i\-ty of $\,g\,$ along $\,\vz$).

The first of these claims is obvious from (\ref{ine}) and (\ref{bcm}). It thus
obviously suffices to show that {\it the second equivalence holds if\/
$\,\vz\,$ is in\-te\-gra\-ble.}

Clearly, with $\,\vz\nh=\mathrm{Ker}\,g\,$ from now on 
assumed in\-te\-gra\-ble,
\begin{equation}\label{wpr}
\begin{array}{l}
N''\nnh=\hs0\,\mathrm{\ if\ and\ only\ if\ 
}\,N''\nh(w,u,v_1\w,\dots,v_r\w)\hs=\hs0\,\mathrm{\ for\ all}\\
\mathrm{local\ vector\ fields\ }\,w,u,v_1\w,\dots,v_r\w\mathrm{\
pro\-ject\-a\-ble\ along\ }\,\vz\nh.
\end{array}
\end{equation}
Although $\,[\hn\Lie\nh g](w,u)\,$ in (\ref{nnh}\hs-b) is not a genuine
$\,1$-form on the manifold $\,M\,$ in question, we now artificially turn it
into one, by fixing a local trivialization of $\,T\nh M\nh$, containing 
a local trivialization of $\,\vz\nh$, and declaring $\,[\hn\Lie\nh g](w,u)\,$
to be $\,1$-form acting by $\,v\mapsto[\hn\Lie\hskip-1.2pt_v\w g](w,u)\,$ on
our selected (finitely many) vector fields $\,v\,$ trivializing
$\,T\nh M\nh$. 
As $\,[\hn\Lie\hskip-1.2pt_v\w g](w,u)=d_v[g(w,u)]-g([v,w],u)-g(w,[v,u])$,
pro\-ject\-a\-bil\-i\-ty of $\,w,u\,$ and (\ref{prj}) imply that 
\begin{equation}\label{lvg}
[\hn\Lie\hskip-1.2pt_v\w g](w,u)=d_v[g(w,u)]\,\mathrm{\ whenever\
}\,v\,\mathrm{\  is\ a\ section\ of\ }\,\vz\nh=\mathrm{Ker}\,g\hh.
\end{equation}
If $\,N''\nnh=\hs0$, the
$\,(r+1)$-form $\,\zeta=N''\nh(w,u,v_1\w,\dots,v_r\w)\,$
vanishes, and hence so does $\,d_v[g(w,u)]\,$ in (\ref{lvg}), for sections
$\,v\,$ of $\,\vz\nh$, as
$\,\zeta(v,\,\cdot\,,\dots,\,\cdot\,)\,$ equals $\,d_v[g(w,u)]\,$ times the
the exterior product 
$\,g(v_1\w,\,\cdot\,)\wedge\ldots\wedge g(v_r\w,\,\cdot\,)\,$ (and the latter 
$\,r$-form
may be chosen nonzero since $\,\mathrm{rank}\,g=r$). Thus, 
by (\ref{prt}), $\,g\,$ is pro\-ject\-a\-ble along $\,\vz$.

Conversely, let us assume pro\-ject\-a\-bil\-i\-ty of $\,g\,$ along
$\,\vz\nh$. Now in (\ref{wpr}) -- (\ref{lvg}) $\,d_v[g(w,u)]=0$, and hence
$\,[\hn\Lie\hskip-1.2pt_v\w g](w,u)=0\,$ for all sections $\,v\,$ of 
$\,\vz\nh$. The $\,1$-form $\,[\hn\Lie\nh g](w,u)\,$
vanishes on $\,\vz\nh=\mathrm{Ker}\,g$, and so obviously do 
$\,g(v_1\w,\,\cdot\,),\dots,g(v_r\w,\,\cdot\,)\,$ in (\ref{wpr}).
As $\,\mathrm{rank}\,g=r$, the $\,1$-forms vanishing on $\,\vz\,$ constitute 
a vector sub\-bun\-dle of fibre dimension $\,r\,$ in
$\,T^*\hskip-1.8ptM\nh$. Thus, $\,N''\nnh=\hs0\,$ by (\ref{wpr}).
\end{proof}

\ \


\ \

\end{document}